\documentclass[10pt,a4paper]{amsart}
\usepackage{amsfonts,layout}
\usepackage{bm,epsfig}

\newcount\minute
\newcount\hour
\def\currenttime{%
	\minute\time
	\hour\minute
	\divide\hour60
	\the\hour:\multiply\hour60\advance\minute-\hour\the\minute}
\def\draftnote{{\it \today \quad  \currenttime \hfill  tex-file :   \jobname}}

\catcode`@=11
\@addtoreset{equation}{section}

\catcode`@=12
\newtheorem{Theorem}{Theorem}[section]
\newtheorem{Definition}{Definition}[section]
\newtheorem{Proposition}{Proposition}[section]
\newtheorem{Lemma}{Lemma}[section]

\newtheorem{Corollary}{Corollary}[section]

\newtheorem{Remark}{Remark}[section]

\newtheorem{Hyp.}{Hyp.}[section]

\begin{document}

\title[]{Parameter determination for Energy Balance Models with Memory}

\author{P. Cannarsa} 
\address{Dipartimento di Matematica, Universit\`a di Roma "Tor Vergata",
Via della Ricerca Scientifica, 00133 Roma, Italy}
\email{cannarsa@mat.uniroma2.it}

\author{M. Malfitana} 
\address{Dipartimento di Matematica, Universit\`a di Roma "Tor Vergata",
Via della Ricerca Scientifica, 00133 Roma, Italy}
\email{martina.malfitana@gmail.com}

\author{P. Martinez} 
\address{Institut de Math\'ematiques de Toulouse; UMR 5219, Universit\'e de Toulouse; CNRS \\ 
UPS IMT F-31062 Toulouse Cedex 9, France} \email{patrick.martinez@math.univ-toulouse.fr}

\thanks{This research was partly supported by Istituto Nazionale di Alta Matematica through the European Research Group GDRE CONEDP. The authors acknowledge the MIUR Excellence Department Project awarded to the Department of Mathematics, University of Rome Tor Vergata, CUP E83C18000100006.}

\begin{abstract}
In this paper, we study two Energy Balance Models with Memory arising in climatology, which consist in a 1D degenerate nonlinear parabolic equation involving a memory term, and possibly a set-valued reaction term (of Sellers type and of Budyko type, in the usual terminology). We provide
existence and regularity results, and obtain uniqueness and stability estimates that are useful for the determination of the insolation function in Sellers' model with memory.
\end{abstract}
\maketitle


\section{Introduction}

\subsection{Energy balance models and the problems we consider} \hfill

We are interested in a problem arising in climatology, coming more specifically from the classical Energy Balance models introduced independently by Budyko \cite{Budyko} and Sellers \cite{sellers}.
These models, which describe the evolution of temperature as the effect of the balance between the amount of energy received from the Sun and radiated from the Earth,  were developed in order to understand the past and future climate and its sensitivity to some relevant parameters on large time scales (centuries). After averaging the surface temperature over longitude, they take the form of the following one-dimensional nonlinear parabolic equation with degenerate diffusion:
\begin{equation*}
	u_t  - (\rho_0 (1-x^2) u_x)_x = R_a - R_e
	\end{equation*}
where
 
\begin{itemize}
\item $u(t,x)$ is the surface temperature averaged over longitude, 
\item the space variable $x = \sin \phi \in (-1,1)$ (here $\phi$ denotes the latitude),
\item $R_a$ represents the fraction of solar energy absorbed by the Earth,
\item $R_e$ represents the energy emitted by the Earth,
\item $\rho_0$ is a positive parameter.
\end{itemize}
A crucial role in the analysis will be played  by  the absorbed energy $R_a$, which is a fraction of the incoming solar flux $Q(t,x)$, that is,
\begin{equation*}
R_a = Q(t,x)\, \beta ,
\end{equation*}
where $\beta$ is the coalbedo function. 
Additionally, as is customary in seasonally averaged models, we will assume that
\begin{equation*}
Q(t,x) = r(t) q(x),
\end{equation*}
where $r$ is positive and $q$ is the so-called "insolation function".

It was noted (see Bhattacharya-Ghil-Vulis \cite{Bhat1982}) that, in order to take into account the long response times that  cryosphere exhibits (for instance, the expansion or retreat of huge continental ice sheets occurs with response times of thousands of years), it is useful to let the coalbedo function depend not only on $u$, but also on the history function, which can be represented by the integral term
	\begin{equation*}
	H(t,x,u):= \int_{-\tau}^{0} k(s,x)u(t+s,x)ds \qquad \forall t>0, x\in I ,
	\end{equation*}
where $k$ is the memory kernel (and  $\tau \sim 10^4$ years, in real problems). As in Roques-Checkroun-Cristofol-Soubeyrand-Ghil \cite{Roques20140349}, we will assume a nonlinear response to memory in the form 
\begin{equation*}
f(H(t,x,u)) .
\end{equation*}
Hence, we are interested in the following Energy Balance Model with Memory (EBBM) problem,
set in the space domain $I:=(-1,1)$:
	\begin{equation*}
	\begin{cases}
	u_t - (\rho _0 (1-x^2) u_x)_x = Q(t,x)\beta (u) + f(H(t,x,u)) - R_e (u), \quad t>0,  x \in I,\\
	\rho _0 (1-x^2) u_x (t,x) = 0, \quad t>0, x \in \partial I\\
	u(s,x) = u_0(s,x), \quad s \in [-\tau,0] .
	\end{cases}
	\end{equation*}
Concerning the function $\beta$, we will assume, as it is classical for such problems, that
\begin{itemize}
\item either $\beta$ is positive and at least Lipschitz continuous (the classical assumption for Sellers type models),
\item or $\beta$ is positive, monotone and discontinuous (the classical assumption for Budyko type models).
\end{itemize}


\subsection{Relation to literature and presentation of our main results} \hfill
\label{sec-plap}

The mathematical analysis of quasilinear EBMM problems of the form
$$ 	\partial_tu - \text{div } (\rho (x) \vert \nabla u \vert ^{p-2} \nabla u) = f(t,x,u,H(t,x,u))
$$
has been the subject of many deep works for a long time. Questions such as well-posedness, uniqueness, asymptotic behavior, existence of periodic solutions, bifurcation, free boundary, numerical approximation were investigated 
for:
\begin{itemize}
\item 1-D models without memory by Ghil in the seminal paper  \cite{Ghil76},
\item 0-D models  in Fraedrich~\cite{Fraedrich78, Fraedrich79}, 
\item 1-D models with memory in Bhattacharya-Ghil-Vulis~\cite{Bhat1982} and Diaz~\cite{Diaz2, Diaz1997}, 
\item 2-D models (on a manifold without boundary, typically representing the Earth's surface) in Hetzer~\cite{Hetzer96globalexistence}, Diaz-Tello~\cite{Diaz-Tello}, Diaz-Hetzer~\cite{Diaz-Hetzer}, Hetzer~\cite{Hetzer1}, Diaz~\cite{DIAZ20062053}, Diaz-Hetzer-Tello~\cite{DHT}, and Hetzer~\cite{hetzer2011global}.
\end{itemize}

In this paper, we are interested in the following inverse problem: is it possible to recover the insolation function (which is a part of the incoming solar flux in $Q(t,x)$) from measurements of the solution, for our EBBM model? Our motivation comes from 
the fact that, with suitable tuning of their parameters, EBMs have shown to mimic the observed zonal temperatures for the observed present climate \cite{north}, and can be used to estimate the temporal response patterns to various forcing scenarios, which is of interest in particular in the detection of climate change. Unfortunately, in practice, the model coefficients cannot be measured directly, but are quantified through the measures of the solution \cite{Roques20140349}. Hence, results proving that measuring the solution in some specific (small) part of the space and time domain is sufficient to recover a specified coefficient are of practical interest.

Several earlier papers are related to this question, in particular the ones that we recall below.
\begin{itemize} 
\item In Tort-Vancostenoble~\cite{TORT2012683}, the question of determining the insolation function was studied for a 1D Sellers type model without memory, combining:
\begin{itemize} 
\item the method introduced  by Imanuvilov-Yamamoto in the seminal paper~\cite{Imanuvilov} (based on the use of Carleman estimates to obtain stability results for the determination of source terms for parabolic equations), 
\item the Carleman estimates from Cannarsa-Martinez-Vancostenoble~\cite{sicon2008} for degenerate parabolic equations,
\item suitable maximum principles to deal with  nonlinear terms.
\end{itemize}
In the same paper, the authors proved stability estimates measuring the solution on an open subset of the space domain. Similar questions were studied in Martinez-Tort-Vancostenoble~\cite{Sellers-PM-JT-JV} on manifolds without boundary.
\item In Roques-Checkroun-Cristofol-Soubeyrand-Ghil~\cite{Roques20140349}, the  question of determining the insolation function was studied for a 1D Sellers type model including memory effects, but for a nondegerate diffusion. These authors extended a method due to Roques-Cristofol~\cite{Roques, Cristofol} which, based on analyticity,  allows for measurements only at a single point $x_0$ (under a rather strong assumption on the kernel appearing in the history function).
\end{itemize}

In this paper we study, first, the 1D Sellers type problem with degenerate diffusion and memory effects.  More precisely, we  prove regularity results  and use them to study the determination of the insolation function, obtaining
\begin{itemize}
\item a uniqueness result, under pointwise observation,
\item a Lipschitz stability result, under localized observation,
\end{itemize}
in the spirit of the aboce mentioned references.
Then, we address 1D Budyko type problems with degenerate diffusion and memory effects, for which we obtain precise existence results as in Diaz-Hetzer~\cite{Diaz-Hetzer}. For this, we need to  regularize the coalbedo and use the existence results obtained in the first part of the paper.

Let us note that our existence results for Sellers and Budyko type problems can be regarded as a consequence of the ones by Diaz-Hetzer~\cite{Diaz-Hetzer} for manifolds. However, here we give a direct proof of such results in zonally averaged 1D settings. For this reason,  we need to use the properties of degenerate diffusion operators.

Finally, to give a more complete overview of the literature on these questions, let us also mention the papers by:
\begin{itemize}
\item Pandolfi~\cite{Pandolfi}, for a similar question but on a different equation (the history function depending on the second-order derivative in space), 
\item Guerrero-Imanuvilov~\cite{Guerrero-Imanuvilov}, that proves that null controllability does {\it not} hold for the linear heat equation perturbed by $\int _0 ^t u$ (hence with a memory term which takes into account all the history from time $0$ to $t$),
\item Tao-Gao \cite{Tao-Gao}, that gives positive null controllability results for a similar heat equation  under additional assumptions on the kernel appearing in the history function (notice however  that these assumptions are incompatible with our settings as they would force the kernel to depend also $t$ and to uniformly vanish at some time $T$, which is unnatural in climate modelling).
\end{itemize}


\section{Mathematical assumptions for these climate models}

We are interested in a class of EBMM:
\begin{equation}
	\begin{cases}
	u_t - (\rho (x) u_x)_x = R_a (t,x,u,H) - R_e (t,x,u,H)), \quad t>0,  x \in I,\\
	\rho (x) u_x = 0, \quad x \in \partial I ,\\
	u(s,x) = u_0(s,x), \quad s \in [-\tau,0], x\in I ,
	\end{cases}
\label{complete eq}
\end{equation}
where $I=(-1,1)$. We are going to precise our assumptions concerning Budyko type problems and Sellers ones.

\subsection{Budyko type models with memory} \hfill

We make the following assumptions:
\begin{itemize}
\item concerning the diffusion coefficient: we assume that there exists $\rho _0 >0$ such that
	\begin{equation}
	\label{hyp-rho}
\forall x\in (-1,1), \quad	\rho (x) := \rho _0 (1-x^2);
	\end{equation}
	
\item concerning $R_a$: we assume that
	\begin{equation}
	\label{hyp-Ra}
	R_a(t,x,u,H) = Q(t,x) \beta(u) + f(H(t,x,u)), 
	\end{equation}
where
\begin{itemize}
\item $Q(t,x)$ is the incoming solar flux; we assume that $Q(t,x)=r(t)q(x)$, where $q$, the insolation function, and $r$ are such that:
		\begin{equation} \label{hyp-q-r} 
		\begin{cases}
		q \in L^\infty(I),  \\
		r \in C^1(\mathbb{R_+}) \text{ and } r,r' \in L^\infty (\Bbb R_+) ;
		\end{cases}
		\end{equation}
\item $\beta$ is the classical Budyko type coalbedo function: it is an highly variable quantity which depends on many local factors such as the cloud cover and the composition of the Earth's atmosphere, moreover it is used as an indicator for ice and snow cover; usually it is considered roughly constant for temperatures far enough from the ice-line, that is a circle of constant latitude that separates the polar ice caps from the lower ice-free latitudes; the classical Budyko type coalbedo is:

\begin{equation}
\label{hyp-Budyko}
\beta (u) = 
\begin{cases}
 a_i ,      & u < \overline{u} , \\
[a_i,a_f] , & u=\overline{u} , \\
 a_f ,      & u > \overline{u} ,
 \end{cases}
\end{equation}
where $a_i < a_f$ (and the threshold temperature $ \bar{u}:= -10^\circ $);
\item $H$ is the history function; it is assumed to be given by 
\begin{equation}
\label{def-H}
	H(t,x,u) = \int _{-\tau} ^0 k(s,x) \, u(t+s,x) \, ds
	\end{equation}
where the kernel $k$ is such that:
		\begin{equation} 
		k \in C^1([-\tau,0] \times [-1,1]; \mathbb{R})
;
		\label{history hp}
		\end{equation}
\item $f$: the nonlinearity that describes the memory effects; we assume that $f: \mathbb{R} \rightarrow \mathbb{R}$ is $C^1$ and such that
		\begin{equation} \label{hyp-f}
		\begin{cases}
		f, f' \in L^{\infty}(\mathbb{R})
		\\
		f,f' \mbox{ are } L-\mbox{Lipschitz} ;
		\end{cases} 
		\end{equation}
\end{itemize}
\item concerning $R_e$: the classical Budyko type assumption is
\begin{equation}
\label{hyp-Re}
R_e (t,x,u,H) = a+bu ,
\end{equation}
where $a$, $b$ are constants;
\item the initial condition: since we define $H$ over a past temperature, the initial condition in such models has to be of the form
	\begin{equation}
	u(s,x) = u_0(s,x) \quad \forall s \in [-\tau,0], \quad x \in I
	\label{initial condition}
	\end{equation}
	for some $u_0(s,x)$ defined on $[-\tau,0] \times I$, for which we will precise our assumptions in our different results.
\end{itemize}

Sometimes we will only add positivity assumptions on $q$ and $r$; these assumptions are natural with respect to the model, but only useful in the inverse problems results.

\subsection{Sellers type models with memory} \hfill

The differences concern the assumptions on the coalbedo and on the emitted energy:
\begin{itemize}
\item $\beta$: in Sellers type models, we assume that 
\begin{equation}
\label{beta-Sellers}
\beta \in C^2(\mathbb{R}), \quad \beta, \beta ', \beta '' \in L^\infty (\Bbb R) 
\end{equation}
(typically, $\beta$ is $C^2$ and takes values  between the lower value for the coalbedo $ a_i $ and higher value $a_f$ (even if there is a sharp transition between these two values around the threshold temperature $ \bar{u}$)).

\item $R_e$ is assumed to follow a Stefan-Boltzmann type law (assuming that the Earth radiates as a black body):
	\begin{equation}
	R_e= \varepsilon(u)|u|^3u ,
	\label{E out s}
	\end{equation}
where the function $\varepsilon$ represents the emissivity; we assume that 
\begin{equation}
\label{epsilon-Sellers}
\begin{cases} 
\varepsilon \in C^1(\mathbb{R}) \text{ and } \varepsilon, \varepsilon ' \in L^\infty (\Bbb R) , \\
\exists \varepsilon _1 >0, \text{ s.t. } \forall u, \quad \varepsilon (u) \geq \varepsilon _1 >0 .
\end{cases}
\end{equation}


\end{itemize}

\subsection{Plan of the paper} \hfill

\begin{itemize}

\item section \ref{s-results-1D} contains the statement of our results concerning Sellers type models:
\begin{itemize}
\item concerning well-posedness questions: see Theorem \ref{thm-Sellers global existence} in section \ref{subsec-wpSellers}; 
\item concerning inverse problems questions: 
\begin{itemize}
\item Theorem \ref{thm-inverse}: uniqueness of the insolation function under pointwise measurements (in section \ref{subsec-inv-pt}),
\item Theorem \ref{thm-inv-stab}: Lipschitz stability under localized measurements (in section \ref{subsec-inv-loc});
\end{itemize}
\end{itemize}

\item section \ref{sec-b-results-1D} contains the statement of our well-posedness result concerning Budyko type models, see Theorem \ref{thm-Budyko global existence};

\item section \ref{sec-open} is devoted to mention some open questions;

\item section \ref{sec-proof-wp-S} contains the proof of Theorem \ref{thm-Sellers global existence};

\item section \ref{inverse problem} contains the proof of Theorem \ref{thm-inverse};

\item section \ref{inverse problem2} contains the proof of Theorem \ref{thm-inv-stab};

\item section \ref{sec-proof-wp-B} contains the proof of Theorem \ref{thm-Budyko global existence}.

\end{itemize}


\section{Main results for the Sellers type model} 
\label{s-results-1D}

First we show the local and global existence of a regular solution to the following problem:
\begin{equation}
\label{stab-eq-u}
	\begin{cases}
	u_t - (\rho(x) u_x)_x = r(t)q(x)\beta (u) -\varepsilon (u) \vert u \vert ^3 u + f(H), \quad t>0,  x \in I,\\
	\rho (x) u_x = 0, \quad t>0, x \in \partial I ,\\
	u(s,x) = u_0(s,x), \quad s \in [-\tau,0], x\in I ,
	\end{cases}
\end{equation}
In this section, we assume \eqref{hyp-rho}, \eqref{hyp-q-r}, \eqref{def-H}-\eqref{hyp-f}, \eqref{initial condition}-\eqref{epsilon-Sellers}. 

In the following, we recast \eqref{stab-eq-u} into a semilinear evolution equation governed by an analytic semigroup.

\subsection{Functional framework} \hfill

	Since the diffusion coefficient has a degeneracy at the boundary, it is necessary to introduce the weighted Sobolev space $V$ below in order to deal the well-posedness of problem \eqref{complete eq}. To know more about this functional framework for one-dimensional degenerate parabolic equations, the reader may also refer to \cite{Campiti1998, cannarsa2005, Cannarsa2008, TORT2012683}.
	
	\begin{equation*}
	V:=\{w \in L^2(I): w \in AC_{loc}(I),  \sqrt{\rho}w_x \in L^2(I)\}
	\end{equation*}
	endowed with the inner product
	\begin{equation*}
	(u,v)_V:=(u,v)_{L^2(I)}+(\sqrt{\rho}u_x,\sqrt{\rho}v_x)_{L^2(I)} \quad \forall u,v \in V
	\end{equation*}
	and then with the associated norm
	\begin{equation*}
	||u||_V:= \sqrt{(u,u)_V} = ||u||_{L^2(I)} + ||\sqrt{\rho}u_x||_{L^2(I)} \quad \forall u \in V .
	\end{equation*}
	
	We recall that $\rho(x) = \rho_0(1-x^2)$ for all $x \in I$ by definition.
	Let us remark that $(V,(\cdot,\cdot)_V)$ is a Hilbert space and that $V \subset H^1_{loc}(I)$ and $V \subset L^2(I) \subset V^*$.
Moreover, 
\begin{itemize}
\item the space $C_0^\infty(I)$ is dense in V, in particular $V$ is dense in $L^2(I)$ (\cite{Campiti1998}, Lemma 2.6));
\item for all $p \in [1,+\infty)$, the inclusion 
\begin{equation}
\label{VLp}
V \hookrightarrow L^p(I)
\end{equation} holds and is continuous; moreover, the inclusion $V \hookrightarrow L^2(I)$ is compact (\cite{Diaz1997}, Lemma 1).
\end{itemize}
	
	In order to obtain our semilinear evolution equation let us define the operator $A:D(A) \subset L^2(I) \rightarrow L^2(I)$ in the following way:
	\begin{equation}
		\begin{cases}
			D(A):=\{u \in V:\rho u_x \in H^1(I) \}\\
			Au:= (\rho u_x)_x \quad u \in D(A)
		\end{cases}
		\label{operator}
	\end{equation}
	
(Note that the boundary condition appearing in \eqref{complete eq} is contained in the definition of the unbounded operator $A$ given in \eqref{operator}: indeed, if $u\in D(A)$, then $\rho u_x \in H^1(I)$, hence $\rho u_x \in C^0(\bar{I})$,
which implies that $\rho u_x \to L$ as $x\to 1^-$; but if $L\neq 0$, then $\sqrt{\rho} u_x \notin L^2(I) $, therefore $L = 0$ and $(\rho u_x)(1) = 0$.
And the case $x = -1$ is analogous.)
	
We denote $\mathcal L (L^2(I))$ the space of linear continuous applications from $L^2(I)$ into itself, endowed with the natural norm $ |||\quad \cdot \quad |||_{\mathcal L (L^2(I))}$. We recall the following
	\begin{Theorem} (\cite{Campiti1998, bensoussan1992})
\cite{Campiti1998}
	\label{thm-Aanal}
		(A,D(A)) is a self-adjoint operator and it is the infinitesimal generator of an analytic and compact semigroup $\{e^{tA}\}_{t \ge 0}$ in $L^2(I)$ that satisfies
		$$ ||| e^{tA} |||_{\mathcal L (L^2(I))} \leq 1. $$
	\end{Theorem}
\noindent (We give elements of its proof in section \ref{pf-thm3.1}).	
Finally, we recall also the following

	\begin{Proposition} (\cite{lions1968problemes}, Proposition 2.1)
		The real interpolation space constructed by the trace method $[D(A),L^2(I)]_{\frac{1}{2}}$ is the space $V$.
		\label{interpolation}
	\end{Proposition}


\subsection{The concept of mild solution for the Sellers type model \eqref{stab-eq-u}} \hfill

Consider the problem \eqref{complete eq} of the Sellers type.
In order to recast it into an abstract form, we introduce the following notations:
\begin{itemize}
\item to manage the nonlinear term, we consider the following function
	\begin{equation}
	G: [0,T] \times V \rightarrow L^2(I), \quad 
G(t,u)(x) = Q(t,x)\beta(u (x)) - \varepsilon(u (x))|u(x)|^3 u(x) ,
	\end{equation}
where we recall that $Q(t,x) = r(t)q(x)$; (note that since $V \hookrightarrow L^p(I)$ for all $p\geq 1$,
it is clear that $G ([0,T] \times V) \subset L^2(I)$);

\item to manage the shifted memory term: 
\begin{itemize}
\item given $u\in C([-\tau,T];L^2(I))$, given $t \in [0,T]$, we consider the right translation $u^{(t)} \in C([-\tau,0];L^2(I))$ by the formula
	\begin{equation}
	\label{eq-shift}
	u^{(t)}:  [-\tau,0] \to L^2(I), \quad u^{(t)}(s) := u(t+s) ,
		\end{equation}
\item and we define the following function
	\begin{equation}
	F: C([-\tau,0];L^2(I)) \rightarrow L^2(I) , \quad 
	F(v)(x) = f \left(\int_{-\tau}^{0}k(s,x) \, ( v(s)(x) ) \, d\sigma \right) ,
	\label{def F}
	\end{equation}
	\end{itemize}
	in such a way that the memory term can be written $F(u^{(t)})$.
	\end{itemize}
And then, given $T>0$, \eqref{stab-eq-u} on $[0,T]$ can be recast into:
	\begin{equation}
	\begin{cases}
	\dot{u}(t) = Au(t) + G(t,u) + F(u^{(t)}) & t \in [0,T]\\
	u(s) = u_0(s) & s \in [-\tau,0] .
	\label{integrodifferential}
	\end{cases}
	\end{equation}

Before defining the concept of mild solution for \eqref{integrodifferential}, we precise the concept of mild solution for the following linear nonhomogeneous problem
\begin{equation}
			\begin{cases}
			\dot{u}(t) = Au(t) + g(t) & t \in [0,T]\\
			u(0) = u_0 .
			\end{cases} 
			\label{linear eq}
			\end{equation}
We consider the following
	\begin{Definition}
		Let $ g \in L^2(0,T;L^2(I))$ and let $u_0 \in L^2(I)$. The function $u \in C([0,T]; L^2(I))$ defined by
		\begin{equation}
		\label{def-form-int}
\forall t\in [0,T], \quad 	u(t) = e^{tA} u_0 + \int_{0}^{t} e^{(t-s)A}g(s)ds
		\end{equation}
is called the \textup{mild solution} of \eqref{linear eq}.
	\end{Definition}
\noindent We recall that $u$ defined by \eqref{def-form-int}
has the following additionnal regularity:
$$ 	u \in H^1(0,T;L^2(I)) \cap L^2(0,T;D(A)).$$

Now we are ready to define the concept of mild solution for \eqref{integrodifferential}:
	
\begin{Definition}
Given $u_0 \in C([-\tau,0];V)$, a function 
$$ u \in H^1(0,T;L^2(I)) \cap L^2(0,T;D(A)) \cap C([-\tau,T];V)$$ 
is called a \textup{mild solution} of \eqref{integrodifferential} on $[0,T]$ if
		\begin{itemize}
			\item[(i)] $u(s) = u_0(s)$ for all $s \in [-\tau,0]$;
			
			\item [(ii)] for all $t\in [0,T]$, we have
			\begin{equation}
			u(t) = e^{tA}u_0(0) + \int_{0}^{t} e^{(t-s)A} \bigl( G(s,u) + F(u^{(s)}) \bigr) \, ds .
			\label{nonlinear eq}
			\end{equation}
		\end{itemize}
		\label{mild}
\end{Definition} 
	

\subsection{Global existence and uniqueness result for the Sellers model \eqref{integrodifferential}} \hfill
\label{subsec-wpSellers}

Now we are ready to prove the global existence result of the integrodifferential problem.
	
	\begin{Theorem}
Consider $u_0$ such that
$$ u_0 \in C([-\tau,0];V) \quad \text{ and } \quad u_0(0) \in D(A) \cap L^\infty (I) .$$
Then, for all $T>0$, the problem \eqref{integrodifferential} has a unique mild solution $u$ on $[0,T]$.
		\label{thm-Sellers global existence}
	\end{Theorem}

(Note that 
\begin{itemize}
\item existence and uniqueness of a global regular solution to \eqref{complete eq} without the memory term has been proved in \cite{TORT2012683}
; 

\item the local existence of our model without the boundary degeneracy has been studied in \cite{Roques20140349};

\item the global existence of a similar 2D-model with memory (hence on a manifold but without the boundary degeneracy), has been investigated in \cite{Diaz-Hetzer}.)
\end{itemize}
	

\subsection{Inverse problem results: determination of the insolation function} \hfill


Here we prove that the insolation function $q(x)$ can be determined in the whole space domain $I$ by using only local information about the temperature.
	
	To achieve this goal, we add the following extra assumptions, as in \cite{Roques20140349}:
the very recent past temperatures are not taken into account in the history function:
	\begin{equation}
	\exists \delta > 0  \mbox{ s.t. } k(s,\cdot) \equiv 0 \quad \forall s \in [-\delta,0]
	\label{hp delta}
	\end{equation}
	where $\delta < \tau$.
(We will discuss about this assumption in section \ref{sec-open}.)

Hence, we have the following situation:
consider two insolation functions $q$ and $\tilde q$, two initial conditions $u_0$ and $\tilde u_0$, and the associated solutions: $u$ satisfying \eqref{stab-eq-u}
and $\tilde u$ satisfying
\begin{equation}
\label{stab-eq-utilde}
	\begin{cases}
	\tilde u_t - (\rho (x) \tilde u_x)_x = r(t)\tilde q(x)\beta (\tilde u) -\varepsilon (\tilde u) \vert \tilde u \vert ^3 \tilde u + f(\tilde H), \quad t>0,  x \in I,\\
	\rho (x) \tilde u_x = 0, \quad x \in \partial I ,\\
	\tilde u(s,x) = \tilde u_0(s,x), \quad s \in [-\tau,0], x\in I ,
	\end{cases}
\end{equation}
where we denote
$$
\tilde H := H(t,x,\tilde u) = \int _{-\tau} ^{-\delta} k(s,x) \tilde u(t+s,x) \, ds .$$
In the following, we state two inverse problems results, according to different assumptions on the control region.

\subsubsection{Pointwise observation and uniqueness result} \hfill
\label{subsec-inv-pt}

Let us choose suitable regularity assumptions on the initial conditions and on the insolation functions, in order to have sufficient regularity on the time derivative of the associated solutions: we consider

\begin{itemize}
\item the set of admissible initial conditions: we consider

\begin{equation}
\label{ensembleci2D-A-pt}
\mathcal U  ^{(pt)} = C^{1,2}([-\tau,0] \times [-1,1]),
\end{equation}

\item and the set of admissible coefficients: we consider 

\begin{equation}
\label{ensembleci2D-coeff-pt}
\mathcal Q ^{(pt)} :=\{q \mbox{ is Lipschitz-continuous and piecewise analytic on } I\} ,
\end{equation}
where we recall the following

	\begin{Definition}
		A continuous function $\psi$ is called \textit{piecewise analytic} if there exist $n \ge 1$ and an increasing sequence $(p_j)_{1 \leq j \leq n}$ such that $p_1 = -1$, $p_n=1$, and
		$$ \psi(x)=\sum_{j=1}^{n-1} \chi_{[p_j,p_{j+1})}(x) \varphi_j(x) \quad \forall x \in I , $$ 
where $\varphi_j$ are analytic functions defined on the intervals $[p_j,p_{j+1}]$ 
and $\chi_{[p_j,p_{j+1})}$ is the characteristic function of the interval $[p_j,p_{j+1})$
for $j = 1, \dots, n-1$.
	\end{Definition}
\end{itemize}
		
Then we prove the following uniqueness result:
	\begin{Theorem}
Consider
\begin{itemize}
\item two insolation functions $q, \tilde{q} \in \mathcal Q ^{(pt)}$ (defined in \eqref{ensembleci2D-coeff-pt})
\item an initial condition $u_0=\tilde u_0 \in \mathcal U  ^{(pt)}$ (defined in \eqref{ensembleci2D-A-pt})
\end{itemize}
and let $u$ be the solution of \eqref{stab-eq-u} and $\tilde{u}$ the solution of \eqref{stab-eq-utilde}. 

Assume that
\begin{itemize}
\item the memory kernel satisfies \eqref{hp delta},
\item $r$ and $\beta$ are positive,
\item there exists $x_0 \in I$ and $T >0$ such that
		\begin{equation}
\forall t \in (0,T), \quad	
\begin{cases}
	u(t,x_0) = \tilde{u}(t,x_0) , \\
	u_x (t,x_0) = \tilde{u}_x (t,x_0) .
	\end{cases}
		\label{assumption u}
		\end{equation}
\end{itemize}
Then $q \equiv \tilde{q}$ on $I$.
		\label{thm-inverse}
\end{Theorem}

This result means that the insolation function $q(x)$ is uniquely determined on $I$ by any measurement of $u$ and $u_x$ at a single point $x_0$ during the time period $(0,T)$. Theorem \ref{thm-inverse} is a natural extension of \cite{Roques20140349} to the degenerate problem.



\subsubsection{Localized observation and stability result} \hfill
\label{subsec-inv-loc}

Let us choose suitable regularity assumptions on the initial conditions and on the insolation functions, in order to have sufficient regularity on the time derivative of the associated solutions: we consider

\begin{itemize}
\item the set of admissible initial conditions: given $M>0$, we consider $\mathcal U ^{(loc)} _M$:

\begin{multline}
\label{ensembleci2D-A}
\mathcal U^{(loc)} _M :=\{u_{0}\in C([-\tau,0]; V\cap L^\infty (-1,1)), u_0(0) \in D(A), Au_0(0) \in L^\infty (I),
\\
\sup _{t\in [-\tau,0]} \bigl( \Vert u_0 (t) \Vert _{V} + \Vert u_0 (t) \Vert _{L^\infty } \bigr) + \Vert A u_0 (0)\Vert _{L^\infty(I)} \leq M \} ,
\end{multline}

\item and the set of admissible coefficients: given $M'>0$, we consider 

\begin{equation}
\label{ensembleci2D-coeff}
\mathcal Q^{(loc)}  _{M'} :=\{q \in L^{\infty}(I): \Vert q \Vert _{L^{\infty}(I)}\leq M'\}.
\end{equation}
\end{itemize}

Now we are ready to state our Lipschitz stability result:
\begin{Theorem}
\label{thm-inv-stab}
Assume that 
\begin{itemize}
\item the memory kernel satisfies \eqref{hp delta},
\item $r$ and $\beta$ are positive. 
\end{itemize}
Consider 
\begin{itemize}
\item $0 < T' < \delta$,
\item $t_0 \in [0, T')$, $T>T'$,
\item $M,M'>0$.
\end{itemize}
Then there exists 
$C (t_0,T',T,M,M')>0$ such that, for all $u_0, \tilde u_0 \in \mathcal U^{(loc)} _M$ (defined in \eqref{ensembleci2D-A}), for all $q, \tilde q \in \mathcal Q^{(loc)}  _{M'}$ (defined in \eqref{ensembleci2D-coeff}),  the solution $u$ of \eqref{stab-eq-u} and the solution $\tilde u$ of \eqref{stab-eq-utilde} satisfy 
\begin{multline}
\label{PISstab1var}
\Vert q -\tilde q \Vert _{L^{2}(I)}^{2}
\leq C
\Bigl( \Vert u( T')- \tilde u( T') \Vert _{D(A)}^{2}
\\
+ \Vert u_t - \tilde u_t \Vert _{L^{2}((t_0,T)\times (a,b))}^{2}
+ \Vert u_0 - \tilde u_0 \Vert _ {C([-\tau,0]; V)} ^2 \Bigr) .
\end{multline} 
\end{Theorem}

 Theorem \ref{thm-inv-stab} is a natural extension of \cite{TORT2012683}.
%
%


\section{Main result for the Budyko type model} 
\label{sec-b-results-1D}

Now we treat the global existence of regular solutions for the Budyko model. In a classical way (see, e.g. Diaz \cite{Diaz1997}), we study the set valued problem
\begin{itemize} 
\item  first regularizing the coalbedo, hence transforming the Budyko type problem into a Sellers one, for which we have a (unique) regular solution, 
\item and then passing to the limit with respect to te regularization parameter.
\end{itemize}

Since $ \beta$ is the graph given in \eqref{hyp-Budyko}, the Budyko type problem has to be understood as
the following differential inclusion problem:
\begin{equation}
\label{stab-eq-u-B}
	\begin{cases}
	u_t - (\rho (x) u_x)_x \in r(t)q(x)\beta (u) -(a+bu) + f (H(u)), \quad t>0,  x \in I,\\
	\rho (x) u_x = 0, \quad x = \pm 1 ,\\
	u(s,x) = u_0(s,x), \quad s \in [-\tau,0], x\in I .
	\end{cases}
\end{equation}
In this section, we assume \eqref{hyp-rho}-\eqref{initial condition}.


\subsection{The notion of mild solutions for the Budyko model \eqref{stab-eq-u-B}} \hfill

Let us define a mild solution for this kind of problem.

\begin{Definition}
Given $u_0 \in C([-\tau,0);V)$, a function 
$$ u \in H^1(0,T;L^2(I)) \cap L^2(0,T;D(A)) \cap C([-\tau,T];V)$$ 
is called a \textit{mild solution} of \eqref{stab-eq-u-B} on $[-\tau,T]$ iff
	\begin{itemize}
		\item $u(s) = u_0(s)$ for all $s \in [-\tau,0]$;
		\item there exists $g \in L^2([0,T];L^2(I))$ such that 
		\begin{itemize}
		\item $u$ satisfies
		\begin{equation}
\forall t\in [0,T], \quad  u(t) = e^{tA} u_0 (0) + \int _0 ^t e^{(t-s)A} g(s) \, ds ,
		\label{linear2}
		\end{equation}
		 \item and $g$ satisfies the inclusion
		 $$g(t,x) \in r(t)q(x)\beta (u(t,x)) -(a+bu(t,x)) + f (H(t,x,u)) \quad \text{ a.e. } (t,x) \in (0,T)\times I .$$
		 \end{itemize}
	\end{itemize}
\end{Definition}

\subsection{Global existence for the Budyko model \eqref{stab-eq-u-B}} \hfill

\begin{Theorem} \label{thm-Budyko global existence}
Assume that 
$$u_0 \in C([-\tau,0],V) \quad \text{ and } \quad u_0(0) \in D(A) \cap L^\infty (I).$$
Then \eqref{stab-eq-u-B} has a mild solution $u$, which is global in time (i.e. defined in $[0,+\infty)$
and mild on $[0,T]$ for all $T>0$).
\end{Theorem}


\section{Open questions}
\label{sec-open}

Let us mention some open questions related to this work.
\begin{itemize}

\item Concerning the Sellers type models and the inverse problems results given in Theorems \ref{thm-inverse} and \ref{thm-inv-stab}: the assumption on the support of the kernel $k$ is crucial, since it allows us to easily get rid of the memory term, and we do not know what can be done without this assumption; hence
\begin{itemize}
\item it would be interesting to weaken this support assumption; it seems reasonable to think that uniqueness and stability results could be obtained in a more general context, and on the other hand, memory terms can sometimes generate problems (see, in particular, \cite{Guerrero-Imanuvilov});
\item even under the support assumption: from a numerical point of view, it would be interesting to weaken some assumptions (in particular on $T$, since our proof is based on $T<\delta$) in order to have better estimates even if $\delta$ is small.
\end{itemize}

\item Concerning the Budyko type models: a solution is obtained by regularization and passage to the limit; (note that, at least for an EBM without memory term, uniqueness of the solution depends on the initial condition, see Diaz \cite{Diaz1997}); several questions are mathematically challenging, in particular to obtain inverse problems results in that setting; one way could be to obtain suitable estimates from the regularized problem, succeeding in avoiding to use the $C^1$ norm of the regularized coalbedo
(unfortunately, the $C^1$ norm of the regularized coalbedo appears in our estimates in the Sellers model).

\item Finally, it would be interesting to obtain inverse problems results for other problems, in particular of the quasilinear type mentionned in the beginning of section \ref{sec-plap}.

\end{itemize}


\section{Proof of Theorem \ref{thm-Sellers global existence}}
\label{sec-proof-wp-S}

\subsection{Elements for Theorem \ref{thm-Aanal}} \hfill
\label{pf-thm3.1}

To prove that $(A,D(A))$ is self-adjoint and it generates a strongly continuous semigroup of contractions in $L^2(I)$, one can refer to Theorem 2.8 in \cite{Campiti1998} (where in our case $b \equiv 0$).
	
Moreover,
		\begin{equation*}
			\int_{0}^{x}\frac{ds}{\rho_0 (1-s^2)} = \frac{1}{2 \rho_0}(\log(1+x) - \log(1-x))
		\end{equation*}
which implies that $ \int_{0}^{x} \frac{ds}{\rho(s)} \in L^1(I)$, 
		and then  one can show that the semigroup generated by $A$ is compact (see Theorem 3.3 in \cite{Campiti1998}).
	
		For the analyticity of the semigroup see Theorem 2.12 in \cite{bensoussan1992} or Theorem 3.6.1 in \cite{tanabe1960}.
			
	
\subsection{Proof of Theorem \ref{thm-Sellers global existence}: local existence of mild solutions} \hfill

In this section, we prove the following

\begin{Proposition}
\label{prop-local existence}
Consider $u_0$ such that
$$ u_0 \in C([-\tau,0];V) \quad \text{ and } \quad u_0 (0) \in D(A) .$$
Then there exists $t^* >0$ such that the problem \eqref{integrodifferential} has a unique mild solution on $[0,t^*]$.
\end{Proposition}

\subsubsection{The functional setting and main tools} \hfill

It will be more practical to have strictly dissipative operators, so let us consider $\tilde{A}:= A -I$, that satisfies:
		\begin{equation*}
		\begin{cases}
		D(\tilde{A}) = D(A)\\
		\tilde{A}u:= (\rho u_x)_x - u, & u \in D(\tilde{A}) :
		\end{cases}
		\end{equation*}
integrating by parts:
		\begin{equation*}
\forall u \in D(\tilde{A}), \quad 		(\tilde{A}u,u)_{L^2(I)} = - \int_{I} (\rho u_x ^2 + u^2)dx \leq -||u||^2_{L^2(I)}  ,
		\end{equation*}
hence $\tilde{A}$ is strictly dissipative.
We will use the following estimates: using Pazy \cite{pazy1983semigroups} (Theorem 6.13, p. 74) with $\alpha = 1/2$
and $\alpha = 3/4$, there exists $c>0$ such that
		\begin{equation}
	\forall t>0, \quad	|||(-\tilde{A})^{1/2} e^{t \tilde{A}}|||_{\mathcal L (L^2(I))} \leq \frac{c}{\sqrt{t}},
		\label{dissipative pazy}
		\end{equation}
and 
		\begin{equation}
	\forall t>0, \quad			|||(-\tilde{A})^{3/4} e^{t \tilde{A}}|||_{\mathcal L (L^2(I))} \leq \frac{c}{t ^{3/4}} .
		\label{dissipative pazy2}
		\end{equation}

Now, consider also $\tilde{G}(t,u):=u + G(t,u)$ so that, adding and substracting $u$ to the equation, the problem \eqref{integrodifferential} is equivalent to
		\begin{equation}
		\begin{cases}
		\dot{u}(t) = \tilde{A}u(t) + \tilde{G}(t,u(t)) + F (u^{(t)} ) \\
		u(0)=u_0(0) .
		\end{cases}
		\end{equation}
By the definitions above, we know that a mild solution of this problem is a function such that
		\begin{equation*}
		u(t)=
		\begin{cases}
		e^{t\tilde{A}}u_0(0) + \int_{0}^{t}e^{(t-s)\tilde{A}}\left[\tilde{G}(s,u(s))+ F (u^{(s)} ) \right] \, ds, &t>0\\
		u_0(t), &t \in [-\tau,0]
		\end{cases}
		\end{equation*}
Then we consider the suitable functional setting:
\begin{itemize}
\item the space of functions
\begin{equation*}
		\mathbb{X}_R:=\bigl \{ v \in C([-\tau,t^*];V) \ | \ 
		\begin{cases}
		\Vert v(t)\Vert _V \leq R \mbox{ } \forall t \in [-\tau,t^*] , \\
		 v(t)=u_0(t) \mbox{ } \forall t \in [-\tau,0]  , 
		 \end{cases}
		 \bigr \} ,
		\end{equation*}
		\item and the associated application 
		$$ \Gamma: \mathbb{X}_R \subset C([-\tau,t^*];V) \rightarrow C([-\tau,t^*];V) $$
		defined by
		\begin{equation*}
\Gamma (u)(t) :=
		\begin{cases}
		e^{t\tilde{A}}u_0(0) + \int_{0}^{t}e^{(t-s)\tilde{A}}\left[ \tilde{G}(s,u(s))+ F( u^{(s)} ) \right] \, ds, &t\in [0,t^*] \\
		u_0(t), &t \in [-\tau,0] .
		\end{cases}
		\end{equation*}
\end{itemize}
In the following, we prove that $\Gamma$ maps $\mathbb{X}_R$ into itself and is a contraction if the parameters
$R$ and $t^*$ are well-chosen. Then the Banach-Caccioppoli fixed point theorem will tell us that $\Gamma$ has a fixed point,
and this fixed point will be a mild solution of our problem.

We will use the following properties of the function $G$:

	\begin{Lemma} (\cite{TORT2012683}, Lemma 3.4)
\label{G well defined}
Take $R>0$.	Then $G$ is well defined on $[0,T] \times V$ with values into $L^2(I)$. Moreover, we have the following estimates:
\begin{itemize}
\item[(i)] there exists $C_R >0$ such that 
\begin{equation}
\label{G borne u}
\begin{cases}
	||G(t,u)||_{L^2(I)} \leq C_R (1+ ||u||_V ), \\
	\forall t \in [0,T], \forall u\in V \text{ s.t. } \Vert u \Vert _V \leq R  ,
\end{cases}
\end{equation}
and 
\begin{equation}
\label{G lip u}
\begin{cases}
	||G(t,u) - G(t,v)||_{L^2(I)} \leq C_R||u - v||_V , \\
	\forall t \in [0,T], \forall u, v \in V \text{ s.t. } \Vert u \Vert _V , \Vert v \Vert _V \leq R  ;
\end{cases}
\end{equation}

			\item[(ii)] there exists $C>0$ such that
\begin{equation}
\label{G lip t}
\begin{cases}
	||G(t,u) - G(t',u)||_{L^2(I)} \leq C|t - t'| , \\
	\forall t,t' \in [0,T] , \forall u \in V \text{ s.t. } \Vert u \Vert _V \leq R .
\end{cases}
\end{equation}
\end{itemize}
\end{Lemma}

These results come directly from \eqref{VLp}.

Concerning the memory term, we have a similar result:
\begin{Lemma}
\label{F well defined}
Take $T>0$.	Then 
$$ \mathcal F: [0,T] \times C([-\tau,T]; L^2 (I)) \to L^2(I), 
\quad \mathcal F (t,u) (x) := f \Bigl( \int _{-\tau} ^0 k(s,x) \, u(t+s) (x) \, ds  \Bigr) $$
 is well defined. Moreover, we have the following estimates:
\begin{itemize}
\item[(i)] there exists $C >0$ such that
\begin{equation}
\label{F borne u}
\begin{cases}
	|| \mathcal F (t,u) ||_{L^2(I)} \leq C \Vert f \Vert _\infty , \\
	\forall t \in [0,T], \forall u\in C([-\tau,T], L^2(I)) ,
\end{cases}
\end{equation}
and 
\begin{equation}
\label{F lip u}
\begin{cases}
	|| \mathcal F(t, u)  -\mathcal  F(t, v) ||_{L^2(I)} \leq C ||u - v||_{C([-\tau,T], L^2(I))} , \\
	\forall t \in [0,T], \forall u, v \in C([-\tau,T], L^2(I))  ;
\end{cases}
\end{equation}

			\item[(ii)] there exists $C>0$ such that
\begin{equation}
\label{F lip t}
\begin{cases}
	||\mathcal F(t,u) - \mathcal F(t',u)||_{L^2(I)} \leq C || u^{(t)} - u^{(t')} ||_{C([-\tau,0], L^2(I))} , \\
	\forall t,t' \in [0,T], \forall u \in C([-\tau,T], L^2(I))  .
\end{cases}
\end{equation}
\end{itemize}
\end{Lemma}


\subsubsection{Step 1: $\Gamma$ maps $\mathbb{X}_R$ into itself if $t^* >0$ is sufficiently small} \hfill

We recall that
\begin{multline*}
\forall t\in [0,t^*], \quad \Gamma (u)(t) =
		e^{t\tilde{A}}u_0(0) + \int_{0}^{t}e^{(t-s)\tilde{A}}\left[ \tilde{G}(s,u(s))+ F( u^{(s)} ) \right] \, ds
		\\
= e^{t\tilde{A}}u_0(0) + \int_{0}^{t}e^{(t-s)\tilde{A}}\left[ \tilde{G}(s,u(s))+ \mathcal F (s,u) \right] \, ds .
\end{multline*}
Denote 
$$ U_1 (t) := e^{t\tilde{A}}u_0(0) ,$$
$$ U_2 (t) :=  \int_{0}^{t}e^{(t-s)\tilde{A}} \tilde{G}(s,u(s)) \, ds ,$$
$$ U_3 (t) :=  \int_{0}^{t}e^{(t-s)\tilde{A}} \mathcal F (s,u) \, ds .$$
Then we claim that
\begin{equation}
\label{uiC}
U_1, U_2, U_3 \in C([0,t^*]; V) , \quad U_1 (0)=u_0(0), \quad U_2 (0)=0 = U_3 (0) .
\end{equation}
Indeed: from standard regularity arguments, it is clear that $U_1(t), U_2 (t), U_3 (t) \in V$ for all $t\in (0,t^*]$.
Moreover, $U_1 (0)=u_0(0)$, $U_2 (0)=0 = U_3 (0)$. It remains to show the continuity. If $t,t+h \in [0,t^*]$, then
\begin{itemize}
\item first 
\begin{multline*} 
\Vert U_1 (t+h) - U_1 (t) \Vert _V 
= \Big \Vert (-\tilde A)^{1/2} (e^{(t+h)\tilde{A}} - e^{t\tilde{A}}) u_0 (0) \Big \Vert _{L^2 (I)} 
\\
= \Big \Vert  (e^{(t+h)\tilde{A}} - e^{t\tilde{A}}) \bigl( (-\tilde A)^{1/2} u_0 (0) \bigr) \Big \Vert _{L^2 (I)}
\\
= \Big \Vert  \Bigl( \int _t ^{t+h} \tilde A e^{\sigma \tilde A} \, d\sigma  \Bigr)  \bigl( (-\tilde A)^{1/2} u_0 (0) \bigr) \Big \Vert _{L^2 (I)}
\\
= \Big \Vert  \Bigl( \int _t ^{t+h} (-\tilde A) ^{3/4} e^{\sigma \tilde A} \, d\sigma  \Bigr)  \bigl( (-\tilde A)^{3/4} u_0 (0) \bigr) \Big \Vert _{L^2 (I)}
\\
\\
\leq  \Bigl( \int _t ^{t+h} \frac{c}{\sigma ^{3/4}} \, d\sigma  \Bigr)  \Vert (-\tilde A)^{3/4} u_0 (0) \Vert _{L^2 (I)}
\\
\leq c \bigl( (t+h) ^{1/4} - t ^{1/4} \bigr) \Vert  u_0 (0) \Vert _{D(A)}
%
\end{multline*}
which gives that $U_1 \in C([0,t^*]; V)$;

\item next, assume that $h>0$, to simplify; then we have
\begin{multline*} 
\Vert U_2 (t+h) - U_2 (t) \Vert _V 
\\
= \Big \Vert (-\tilde A)^{1/2} \Bigl( \int_{0}^{t+h}e^{(t+h-s)\tilde{A}} \tilde{G}(s,u(s)) \, ds - \int_{0}^{t}e^{(t-s)\tilde{A}} \tilde{G}(s,u(s)) \, ds \Bigr) \Big \Vert _{L^2 (I)} 
\\
= \Big \Vert (-\tilde A)^{1/2} (e^{h\tilde{A}} - Id ) \int_{0}^{t}e^{(t-s)\tilde{A}} \tilde{G}(s,u(s)) \, ds
\\
+ (-\tilde A)^{1/2} \int_{t}^{t+h}e^{(t+h-s)\tilde{A}} \tilde{G}(s,u(s)) \, ds \Big \Vert _{L^2 (I)}
\\
\leq \Big \Vert  (e^{h\tilde{A}} - Id ) \int_{0}^{t} (-\tilde A)^{1/2} e^{(t-s)\tilde{A}} \tilde{G}(s,u(s)) \, ds \Big \Vert _{L^2 (I)}
\\
+ \Big \Vert \int_{t}^{t+h} (-\tilde A)^{1/2} e^{(t+h-s)\tilde{A}} \tilde{G}(s,u(s)) \, ds \Big \Vert _{L^2 (I)}
;
\end{multline*}
then, using once again \eqref{dissipative pazy2}, we have
\begin{multline*}
\Big \Vert  (e^{h\tilde{A}} - Id ) \int_{0}^{t} (-\tilde A)^{1/2} e^{(t-s)\tilde{A}} \tilde{G}(s,u(s)) \, ds \Big \Vert _{L^2 (I)}
\\
= \Big \Vert  \Bigl( \int _0 ^{h} \tilde A e^{\tau \tilde A} \, d\tau \Bigr) \int_{0}^{t} (-\tilde A)^{1/2} e^{(t-s)\tilde{A}} \tilde{G}(s,u(s)) \, ds \Big \Vert _{L^2 (I)}
\\
= \Big \Vert  \Bigl( \int _0 ^{h} (-\tilde A)^{3/4} e^{\tau \tilde A} \, d\tau \Bigr) \Bigl( \int _0 ^{t} (-\tilde A)^{3/4} e^{(t-s)\tilde A} \tilde G(s,u(s)) \, ds \Bigr) \Big \Vert _{L^2 (I)}
\\
\leq C \Bigl( \int _0 ^{h} \frac{1}{\tau ^{3/4}} \, d\tau \Bigr)
 \Bigl( \int _0 ^{t} \frac{1}{(t-s) ^{3/4}} C_R (1+R) \, ds \Bigr)
\\
 = 16C C_R (1+R) h ^{1/4} t^{1/4} ;
\end{multline*}
in the same way, using \eqref{dissipative pazy} we have
\begin{multline*}
\Vert \int_{t}^{t+h} (-\tilde A)^{1/2} e^{(t+h-s)\tilde{A}} \tilde{G}(s,u(s)) \, ds \Vert _{L^2 (I)}
\\
\leq \int_{t}^{t+h} \frac{c}{\sqrt{t+h-s}} \Vert \tilde{G}(s,u(s)) \Vert _{L^2 (I)} \, ds
\\
\leq C_R (1+ R) \int_{t}^{t+h} \frac{c}{\sqrt{t+h-s}} \, ds
= C_R (1+ R) O (\sqrt{\vert h \vert }) ;
\end{multline*}
hence
$$
\Vert U_2 (t+h) - U_2 (t) \Vert _V 
\leq C_R (1+ R) O (\vert h \vert ^{1/4}) ,
$$
which gives that $u_2 \in C([0,t^*]; V)$;

\item finally, using \eqref{dissipative pazy}, \eqref{dissipative pazy2} and \eqref{F borne u} we have (still assuming that $h>0$, in order to simplify)
\begin{multline*} 
\Vert U_3 (t+h) - U_3 (t) \Vert _V 
\\
=\Big \Vert (-\tilde A)^{1/2} \Bigl( \int_{0}^{t+h}e^{(t+h-s)\tilde{A}} \mathcal F (s,u)  \, ds - \int_{0}^{t}e^{(t-s)\tilde{A}} \mathcal F (s,u)  \, ds \Bigr) \Big \Vert _{L^2 (I)} 
\\
\leq \Big \Vert  (e^{h\tilde{A}} - Id ) \int_{0}^{t} (-\tilde A)^{1/2} e^{(t-s)\tilde{A}} \mathcal F (s,u)  \, ds \Big \Vert _{L^2 (I)}
\\
+ \Big \Vert \int_{t}^{t+h} (-\tilde A)^{1/2} e^{(t+h-s)\tilde{A}} \mathcal F (s,u)  \, ds \Big \Vert _{L^2 (I)}
\\
\leq \Vert f \Vert _\infty O (\vert h \vert ^{1/4}) ,
\end{multline*}
which gives that $U_3 \in C([0,t^*]; V)$.

\end{itemize}
We conclude that $\Gamma (u) \in C([0,t^*]; V)$, and since $\Gamma (u) (0^+) = \Gamma (u) (0^-)$, we have that $\Gamma (u) \in C([-\tau,t^*]; V)$.

From the previous study, we also have that
\begin{multline*} 
\forall t \in [0,t^*] , \quad \Vert \Gamma u(t) \Vert _V 
\leq \Vert U_1(t) \Vert _V + \Vert  U_2(t) \Vert _V  + \Vert U_3(t) \Vert _V 
\\
\leq  \Vert u_0 (0) \Vert _V + \int _0 ^t \frac{c}{\sqrt{t-s}} \bigl(  C_R (1+R) + \Vert f \Vert _\infty \bigr) \, ds
\\
= \Vert u_0 (0) \Vert _V + \bigl(  C_R (1+R) + \Vert f \Vert _\infty \bigr) O (\sqrt{t^*}) .
\end{multline*}
Choose, e.g., 
$$ R = \Vert u_0  \Vert _{C[-\tau,0];V)} + 1 ;$$
then choosing $t^*$ small enough, we have that
$$ \forall t \in [-\tau,t^*] , \quad \Vert \Gamma u(t) \Vert _V \leq R ,$$
hence $\mathbb{X}_R$ is stable under the action of $\Gamma$ if $t>0$ is small enough.
		

\subsubsection{Step 2: $\Gamma$ is a contraction if $t^* \in (0,1)$ is sufficiently small} \hfill

Of course,
$$ \forall u,v \in \mathbb{X}_R, \forall t\in [-\tau,0], \quad \Gamma (u) (t) = \Gamma (v) (t) .$$
So we study the difference $\Gamma (u) (t) - \Gamma (v) (t)$ for $t\in [0,t^*]$. As we did previously, we have

\begin{multline*}
\forall t\in [0,t^*], \quad \Gamma (u)(t) - \Gamma (v)(t) 
= \int_{0}^{t}e^{(t-s)\tilde{A}}\left[ \tilde{G}(s,u(s)) - \tilde{G}(s,v(s)) \right] \, ds
\\
+ \int_{0}^{t}e^{(t-s)\tilde{A}}\left[  \mathcal F (s,u) - \mathcal F (s,v) \right] \, ds .
\end{multline*}
Then, using \eqref{dissipative pazy} and \eqref{G lip u}, we have
\begin{multline*}
\Big \Vert \int_{0}^{t}e^{(t-s)\tilde{A}}\left[ \tilde{G}(s,u(s)) - \tilde{G}(s,v(s)) \right] \, ds \Big \Vert _V
\\
= \Big \Vert \int_{0}^{t} (-\tilde A) ^{1/2} e^{(t-s)\tilde{A}} \left[ \tilde{G}(s,u(s)) - \tilde{G}(s,v(s)) \right] \, ds \Big \Vert _{L^2(I)}
\\
\leq \int_{0}^{t} \frac{c}{\sqrt{t-s}} C_R \Vert u(s)-v(s) \Vert _V \, ds 
\\
\leq C_R \Vert u-v \Vert _{C([-\tau,T];V)} \int_{0}^{t} \frac{c}{\sqrt{t-s}}  \, ds 
\\
\leq  2 c C_R \sqrt{t^*} \Vert u-v \Vert _{C([-\tau,T];V)} .
\end{multline*}
And using \eqref{dissipative pazy} and \eqref{F lip u}, we have
\begin{multline*}
\Big \Vert \int_{0}^{t}e^{(t-s)\tilde{A}}\left[  \mathcal F (s,u) - \mathcal F (s,v)  \right] \, ds \Big \Vert _V
\\
= \Big \Vert \int_{0}^{t} (-\tilde A) ^{1/2} e^{(t-s)\tilde{A}} \left[  \mathcal F (s,u) - \mathcal F (s,v)  \right] \, ds \Big \Vert _{L^2(I)}
\\
\leq \int_{0}^{t} \frac{c}{\sqrt{t-s}} C \Vert u-v \Vert _{C([-\tau,T];V)} \, ds 
\\
\leq  2 c C_R \sqrt{t^*} \Vert u-v \Vert _{C([-\tau,T];V)} .
\end{multline*}
We obtain that
$$
\forall t\in [0,t^*], \quad \Vert \Gamma (u)(t) - \Gamma (v)(t) \Vert _V \leq 4 c C_R \sqrt{t^*} \Vert u-v \Vert _{C([-\tau,T];V)} ,$$
hence $\Gamma$ is a contraction if $t^*$ is small enough.


\subsubsection{Step 3: Additional regularity of the solution and conclusion of the proof of Proposition \ref{prop-local existence}} \hfill

		Let us recall that $u \in C([-\tau,t^*];V)$ and
		\begin{equation*}
		u(t) = e^{t\tilde{A}}u_0(0) 
		+ \int_{0}^{t}e^{(t-s)\tilde{A}}\tilde{G}(s,u(s))ds
		+ \int_{0}^{t}e^{(t-s)\tilde{A}} \mathcal F (s,u) \, ds .
		\end{equation*}
		Let us remark that
$$
		t \mapsto \int_{0}^{t}e^{(t-s)\tilde{A}}\tilde{G}(s,u(s))ds \in H^1(0,t^*;L^2(I)) \cap L^2(0,t^*;D(A))
$$
and
$$
t \mapsto \int_{0}^{t}e^{(t-s)\tilde{A}}\mathcal F (s,u) \, ds \in H^1(0,t^*;L^2(I)) \cap L^2(0,t^*;D(A))
$$
using standard regularity results.
%
		Then we can conclude that 
		$$u \in H^1(0,t^*;L^2(I)) \cap L^2(0,t^*;D(A)) \cap C([-\tau,t^*];V) .$$
		 This concludes the proof of Proposition \ref{prop-local existence}. \qed


\subsection{Proof of Theorem \ref{thm-Sellers global existence}: uniqueness} \hfill

In this section we prove the following

\begin{Proposition}
\label{prop-unicite}
Given $T_0>0$ and $u_0 \in C([-\tau,0];V)$, assume that $u$ and $\tilde u$ are mild solutions of the problem \eqref{integrodifferential} on  $[0,T_0]$. Then $u=\tilde u$ on $[0,T_0]$.
\end{Proposition}

\begin{proof}
Consider $w=u-\tilde u$. Then $w$ solves
$$ \begin{cases}
w_t - (\rho w_x)_x = G(t,u) - G(t,\tilde u) + f(H) - f(\tilde H) , \quad t\in (0,T_0), x\in (-1,1) ,\\
\rho w_x = 0 , \quad t\in (0,T_0), x = \pm 1 , \\
w(s,x) = 0 , \quad s\in [-\tau,0], x \in (-1,1) .
\end{cases}$$
Take $T \in (0,T_0)$. Multiplying the first equation by $w$ and integrating on $(0,T)\times (-1,1)$, we obtain
$$ \int _0 ^T \int _{-1} ^1 w\bigl(w_t - (\rho w_x)_x \bigr) = \int _0 ^T \int _{-1} ^1 w\bigl(G(t,u) - G(t,\tilde u) + f(H) - f(\tilde H)\bigr) .$$
Integrating by parts, we have
$$ \int _0 ^T \int _{-1} ^1 w w_t \, dx \, dt = \frac{1}{2} \Vert w(T) \Vert _{L^2 (-1,1)} ^2 ,$$
$$ \int _0 ^T \int _{-1} ^1 w(- (\rho w_x)_x)) \, dx \, dt \geq 0 ,$$
$$ \int _0 ^T \int _{-1} ^1 w\bigl(G(t,u) - G(t,\tilde u)\bigr) \, dx \, dt
\leq c \int _0 ^T \int _{-1} ^1 w(u-\tilde u)\, dx \, dt  = c \int _0 ^T \Vert w(t) \Vert _{L^2 (-1,1)} ^2 \, dt ,$$
\begin{multline*}
\int _0 ^T \int _{-1} ^1 w\bigl(f(H) - f(\tilde H)\bigr) \, dx \, dt
\leq c \int _0 ^T \int _{-1} ^1 \vert w \vert \vert H - \tilde H \vert \, dx \, dt
\\
\leq c \int _0 ^T \Vert w(t) \Vert _{L^2 (-1,1)} ^2 \, dt + c \int _0 ^T \int _{-1} ^1 \vert H - \tilde H \vert ^2 \, dx \, dt .
\end{multline*}
Let us introduce 
$$ W (T') := \int _0 ^{T'} \Vert w(T) \Vert _{L^2 (-1,1)} ^2 \, dT .$$
Using the previous estimates, we have
\begin{equation}
\label{uni-1}
\frac{1}{2} W' (T) \leq 2c W(T) + c \int _0 ^T \int _{-1} ^1 \vert H - \tilde H \vert ^2 \, dx \, dt .
\end{equation}
Concerning the last term:
\begin{multline*}
\int _0 ^T \int _{-1} ^1 \vert H - \tilde H \vert ^2 \, dx \, dt
= \int _0 ^T \int _{-1} ^1 \Big \vert \int _{-\tau} ^0 k(s,x) (u(t+s,x)- \tilde u (t+s,x)) \, ds  \Big \vert ^2 \, dx \, dt
\\
\leq c \int _0 ^T \int _{-1} ^1  \Bigl( \int _{-\tau} ^0 (u(t+s,x)- \tilde u (t+s,x))^2 \, ds \Bigr) \, dx \, dt ;
\end{multline*}
but
\begin{multline*}
 \int _0 ^T \int _{-1} ^1  \Bigl( \int _{-\tau} ^0 (u(t+s,x)- \tilde u (t+s,x))^2 \, ds \Bigr) \, dx \, dt
\\
= \int _0 ^T \int _{-\tau} ^0 \Vert w(t+s) \Vert _{L^2 (-1,1)} ^2 \, ds \, dt 
=  \int _0 ^T \int _{t-\tau} ^t \Vert w(\sigma) \Vert _{L^2 (-1,1)} ^2 \, d\sigma \, dt .
\end{multline*}
Note that the initial condition of $w$ gives us that
$$ t-\tau \leq 0 \quad \implies \quad  \int _{t-\tau} ^t \Vert w(\sigma) \Vert _{L^2 (-1,1)} ^2 \, d\sigma
=  \int _0 ^t \Vert w(\sigma) \Vert _{L^2 (-1,1)} ^2 \, d\sigma ,$$
and of course, 
$$ t-\tau \geq 0 \quad \implies \quad  \int _{t-\tau} ^t \Vert w(\sigma) \Vert _{L^2 (-1,1)} ^2 \, d\sigma
\leq  \int _0 ^t \Vert w(\sigma) \Vert _{L^2 (-1,1)} ^2 \, d\sigma .$$
Hence
$$ \int _0 ^T \int _{t-\tau} ^t \Vert w(\sigma) \Vert _{L^2 (-1,1)} ^2 \, d\sigma \, dt
\leq  \int _0 ^T \int _{0} ^t \Vert w(\sigma) \Vert _{L^2 (-1,1)} ^2 \, d\sigma \, dt ,$$
which gives that
$$ \int _0 ^T \int _{-1} ^1 \vert H - \tilde H \vert ^2 \, dx \, dt
\leq c \int _0 ^T \int _{0} ^t \Vert w(\sigma) \Vert _{L^2 (-1,1)} ^2 \, d\sigma \, dt 
= c \int _0 ^T  W(t) \, dt .$$
Since $W$ is nondecreasing, we obtain that
\begin{equation}
\label{uni-2}
\int _0 ^T \int _{-1} ^1 \vert H - \tilde H \vert ^2 \, dx \, dt \leq cT W(T) ,
\end{equation}
and then we are in position to conclude: we deduce from \eqref{uni-1} and \eqref{uni-2} that
$$ \frac{1}{2} W' (T) \leq (2c + c'T) W(T) .$$
Finally, integrating with respect to $T \in (0,T')$, and using that $W(0)=0$, we obtain that
$$ W(T') \leq 2 \int _0 ^{T'} (2c + c'T) W(T) \, dT .$$
Then Gronwall's lemma tells us that $W=0$, and then $W'=0$, which gives that $w=0$. This concludes the proof of Proposition \ref{prop-unicite}. 

\end{proof}


\subsection{Proof of Theorem \ref{thm-Sellers global existence}: the maximal solution is global in time} \hfill
\label{global}

Proposition \ref{prop-unicite} is a standard uniqueness result. As a consequence, combining it with the local existence result given in Proposition \ref{prop-local existence}, we are able to define the maximal existence time:
\begin{equation}
\label{T*(u_0)}
T^* (u_0) := \sup \{T \geq 0 \text{ s.t.  \eqref{integrodifferential} has a mild solution on  $[0,T]$} \},
\end{equation}
Then, for all $T< T^* (u_0)$, we have a mild solution $u_T$ on $[0,T]$, and 
$$0 < T < T' < T^* (u_0) \quad \implies \quad u _ T = u _{T'} \text{ on } [0,T] ,$$
and this allows us to define the associated maximal solution, defined exactly on $[0, T^*(u_0))$.
It remains to prove the global existence of the maximal solution:

\begin{Proposition}
\label{global existence sellers}
Consider $u_0 \in C([-\tau,0];V)$ and such that $u_0(0) \in D(A) \cap L^\infty (I)$, and the associated maximal mild solution,
defined in $[0,T^* (u_0))$. Then $T^* (u_0)=+\infty$. 
\end{Proposition}
Proposition \ref{global existence sellers} allows us to conclude the proof of  Theorem \ref{thm-Sellers global existence}.
So it remains to prove Proposition \ref{global existence sellers}, and for this we begin by proving a boundedness property.

\subsubsection{$L^\infty$ bound of the solution} \hfill

	First of all, let us prove a general and useful boundedness property. Following the proofs developed by \cite{TORT2012683}, without the memory term, we first state a known preliminary result (see Lemma 6.1 in \cite{TORT2012683}):
	\begin{Lemma}
		Let $u \in V$. Then, for all $M$, $(u-M)^+:= \sup (u-M,0) \in V$ and $(u+M)^-:= \sup (-(u+M),0) \in V$. Moreover for a.e. $x \in I$
		\begin{equation}
		((u-M)^+)_x (x) =
		\begin{cases}
		u_x (x) & (u-M)(x) >0\\
		0 & (u-M)(x) \leq 0 
		\end{cases}
		\label{65}
		\end{equation}
		and for a.e. $x \in I$
		\begin{equation*}
		((u+M)^-)_x(x) =
		\begin{cases}
		0 & (u+M)(x) >0\\
		-u_x (x) & (u+M)(x) \leq 0 
		\end{cases}
		\label{66}
		\end{equation*}
	\end{Lemma}

	Then we can prove the theorem below, where we just add the memory term to the proof of Theorem 3.3 in \cite{TORT2012683}:
	\begin{Theorem}
		Consider $u_0 \in C([-\tau,0];V)$ and $u_0(0) \in D(A) \cap L^\infty (I)$, $T>0$ and $u$ a mild solution of \eqref{integrodifferential} defined on $[0,T]$.
%
Let us denote
$$ M_1 := \left( \frac{||q||_{L^\infty (I)} ||r||_{L^\infty (\mathbb{R})}||\beta||_{L^\infty (\mathbb{R})}+||f||_{L^\infty(\mathbb{R})} }{\varepsilon_1} \right)^{\frac{1}{4}}$$
and 
$$
			M:= \max \{ ||u_0(0)||_{L^\infty (I)}, M_1\}
$$
Then $u$ satisfies
		\begin{equation}
			||u||_{L^\infty((0,T) \times I)} \leq M.
		\end{equation}
		\label{maximum}
	\end{Theorem}
	\begin{proof}
		Let us set $\mathcal{B}:=\{x \in I: u(t,x) > M \}$ and multiply the equation satisfied by $u$ by $(u-M)^+$, then we get the equation below using the previous Lemma and the boundary conditions satisfied by $u$.
		\begin{align*}
		\int_{I} u_t (u-M)^+ dx &+ \int_{I} \rho (((u-M)^+)^2)_x dx
		\\  
		&= \int_{I}[Q\beta(u) - R_e(u)+F(u^{(t)})](u-M)^+ dx\\
		& = \int_{\mathcal{B}}[Q\beta(u) - R_e(u)+F(u^{(t)})](u-M) dx .
		\end{align*}
		Moreover, for $ x \in \mathcal{B} $,
		\begin{equation}
		Q\beta(u) - \varepsilon(u)u|u|^3 +F(u^{(t)}) \leq ||Q||_{\infty} ||\beta||_{\infty} -\varepsilon_1 M^4 + ||f||_{\infty} \leq 0
		\label{leq 0}
		\end{equation}
		thanks to our definition of $M$ and to the assumptions on $F$.
		Then
		\begin{equation*}
		\frac{1}{2} \frac{d}{dt} \int_{I} |(u-M)^+|^2 dx = \int_{I} u_t (u-M)^+ dx \leq 0.	
		\end{equation*}
		for all $t \in [0,T]$.
		Therefore $t \mapsto ||(u-M)^+(t)||^2_{L^2(I)}$ is nonincreasing on $[0,T]$. Since $(u_0(0) - M)^+ \equiv 0 $, we obtain that $u(t,x) \leq M$ for all $t \in [0,T]$ and for a.e. $x \in I$.
		
		In the same way, we can multiply the equation \eqref{complete eq} by $(u+M)^-$ and then we obtain
		\begin{equation*}
		\frac{d}{dt} \int_{I} |(u+M)^-|^2 dx \leq 0.
		\end{equation*} 
		Finally, since $(u_0(0)+M)^- \equiv 0$ we have that $u(t,x) \geq -M$ for all $t \in [0,T]$ and a.e. $x \in I$.
	\end{proof}


\subsubsection{Proof of Proposition \ref{global existence sellers}.} \hfill

	From the previous theorem one may deduce that, for any $u_0(0) \in D(A) \cap L^\infty (I)$, the $L^\infty$-norm of the solution remains bounded on $[0,T^*(u_0))$. To ensure the global existence of the mild solution for the Sellers-type model, we argue by contradiction: we are going to prove that, if $T^* (u_0) < + \infty$, then $t\mapsto u(t)$ can be extended up to $T^*(u_0)$ (and then further), which will be in contradiction with the maximality of $T^* (u_0)$. 

Let us assume that $T^* (u_0) < +\infty$. Note that since $u_0(0) \in  D(A) \cap L^\infty (I)$, Theorem \ref{maximum} implies that $||u||_{L^\infty((0,T^*(u_0)) \times I)} \leq M$. It follows that
		\begin{equation*}
		R_e(u)= \varepsilon(u)u |u|^3 \leq ||\varepsilon|| _{L^\infty(\mathbb{R})} M^4 =: C
		\end{equation*}
		
		It remains to prove that there exists $\lim\limits_{t \uparrow T^*(u_0)} u(t)$ in $V$.
		To prove this, we prove that the function $t\mapsto u(t)$ satisfies the Cauchy criterion. 
		By definition of mild solution, we have
		$$ u(t) = U_1(t) + U_2(t) + U_3(t) $$
		with
		\begin{align*}
		&U_1(t):=e^{t \tilde A}u_0(0) ,\\
		&U_2(t):=\int_{0}^{t}e^{(t-s)\tilde A} \tilde G(s,u(s))ds ,\\
		&U_3(t):= \int_{0}^{t}e^{(t-s)A} \mathcal F (s,u) \, ds .
		\end{align*}
Now, first, $U_1$ has a limit in $V$: $U_1(t) \rightarrow U_1(t^*)$ as $t \rightarrow t^*$ since it is a semigroup applied to the initial value.
		
		
Next, let us prove that $U_2$ has also a limit in $V$ as $t\to t^*$: if $t'\leq t < t^*$, we have
\begin{multline*}
\Vert U_2 (t) - U_2 (t') \Vert _V
= \Big \Vert \int _0 ^t e^{(t-s) \tilde A} \tilde G(s,u(s)) \, ds 
- \int _0 ^{t'} e^{(t'-s)\tilde A} \tilde G(s,u(s)) \, ds \Big \Vert _V 
\\
\leq \Big \Vert \int _0 ^{t'} e^{(t-s)\tilde A} \tilde G(s,u(s)) - e^{(t'-s)\tilde A} \tilde G(s,u(s)) \, ds \Big \Vert _V
+ \Big  \Vert \int _{t'} ^t e^{(t-s)\tilde A} \tilde G(s,u(s)) \, ds \Big \Vert _V .
\end{multline*}
We study these two last terms, proving that they satisfy the Cauchy criterion, hence both will have a limit in $V$:
\begin{itemize}
\item first the last one: using Theorem \ref{maximum}, we have
for all $s \in [0,T^* (u_0))$:
\begin{multline*}
\Vert \tilde G(s,u(s))\Vert _{L^2(I)} 
= \Vert r(s) q(x)\beta(u) - \varepsilon(u)u|u|^3 \Vert  _{L^2(I)}
\\
\leq 2 \Bigl( \Vert r \Vert _{\infty} \Vert q \Vert _{\infty} \Vert \beta \Vert _{\infty} + \Vert \varepsilon \Vert _{\infty} M^4 \Bigr) ,
\end{multline*}
hence
\begin{multline*} 
\Big \Vert \int _{t'} ^t e^{(t-s)\tilde A} \tilde G(s,u(s)) \, ds \Big \Vert _V
= \Big \Vert (-\tilde A)^{1/2} \int _{t'} ^t e^{(t-s)\tilde A} \tilde G(s,u(s)) \, ds \Big  \Vert _{L^2 (I)}
\\
\leq \int _{t'} ^t \Big \Vert (-\tilde A)^{1/2} e^{(t-s)\tilde A} \tilde G(s,u(s)) \Big \Vert _{L^2 (I)} \, ds 
\leq \int _{t'} ^t \frac{C}{\sqrt{t-s}} \, ds = 2 C \sqrt{t-t'} ;
\end{multline*}
\item for the other term:
\begin{multline*}
\Big \Vert \int _0 ^{t'} e^{(t-s)A} G(s,u(s)) - e^{(t'-s)A} G(s,u(s)) \, ds \Big \Vert _V
\\
= \Big \Vert (e^{(t-t')A} - Id) \int _0 ^{t'} e^{(t'-s)A} G(s,u(s)) \, ds \Big \Vert _V
\\
= \Big \Vert (-A)^{1/2} (e^{(t-t')A} - Id) \int _0 ^{t'} e^{(t'-s)A} G(s,u(s)) \, ds \Big \Vert _{L^2 (I)}
\\ 
= \Big \Vert  (e^{(t-t')A} - Id) \int _0 ^{t'} (-A)^{1/2} e^{(t'-s)A} G(s,u(s)) \, ds \Big \Vert _{L^2 (I)}
\\
= \Big \Vert  \Bigl( \int _0 ^{t-t'} A e^{\tau A} \, d\tau \Bigr) \Bigl( \int _0 ^{t'} (-A)^{1/2} e^{(t'-s)A} G(s,u(s)) \, ds \Bigr) \Big \Vert _{L^2 (I)}
\\
= \Big \Vert  \Bigl( \int _0 ^{t-t'} (-A)^{3/4} e^{\tau A} \, d\tau \Bigr) \Bigl( \int _0 ^{t'} (-A)^{3/4} e^{(t'-s)A} G(s,u(s)) \, ds \Bigr) \Big \Vert _{L^2 (I)} ;
\end{multline*}
using \eqref{dissipative pazy2} and the fact that $G$ is bounded, we obtain that
\begin{multline*}
\Big \Vert  \Bigl( \int _0 ^{t-t'} (-A)^{3/4} e^{\tau A} \, d\tau \Bigr) \Bigl( \int _0 ^{t'} (-A)^{3/4} e^{(t'-s)A} G(s,u(s)) \, ds \Bigr) \Big \Vert _{L^2 (I)} 
\\
\leq C \Bigl( \int _0 ^{t-t'} \frac{1}{\tau ^{3/4}} \, d\tau \Bigr)
 \Bigl( \int _0 ^{t'} \frac{1}{(t'-s) ^{3/4}} \, ds \Bigr)
 = C' (t-t') ^{1/4} (t')^{1/4} .
\end{multline*}
\end{itemize}
From these two estimates, we deduce that $t\mapsto U_2 (t)$ satisfies the Cauchy criterion and has a limit as $t\to T^* (u_0)$ if $T^* (u_0) < +\infty$. 

In the same way, $t\mapsto U_3 (t)$ satisfies the Cauchy criterion as $t\to T^* (u_0)$. It follows that $t\mapsto u(t)$ has a limit as $t \rightarrow T^* (u_0)$ if $T^* (u_0) < +\infty$, which contradicts the maximality of $T^* (u_0)$, hence $T^* (u_0) = +\infty$. \qed

\begin{Remark}{\rm
	If we assume that $R_e$ is linear as in the Budyko models, i.e. $R_e(u) = A + Bu$, 
	everything remains true, and  we can apply the fixed point theorem and extend the solution to a global one.}
	\label{local budyko}
\end{Remark}


	\section{Proof of Theorem \ref{thm-inverse}}
	\label{inverse problem}

		
Without loss of generality, we can assume that 
		\begin{equation*}
 		0 < T < \delta  .
		\label{to extend}
		\end{equation*}

Using the extra assumption \eqref{hp delta}, we see that the history term satisfies
		\begin{multline*}
\forall t\in (0,T), \quad		H(t,x,u) 
		= \int_{-\tau}^{0} k(s,x)u(t+s,x) ds
		= \int_{-\tau}^{-\delta} k(s,x)u(t+s,x) ds
		\\
		= \int_{t-\tau}^{t-\delta} k(\sigma-t,x)u(\sigma,x) d\sigma
		= \int_{t-\tau}^{t-\delta} k(\sigma-t,x)u_0(\sigma,x) d\sigma = H(t,x,u_0) ,
		\end{multline*}
		where we used \eqref{initial condition} and that $[t-\tau,t-\delta] \subset [-\tau,0]$ since $t \leq T< \delta$.
		
Hence during this small interval of time, the memory term depends only on the initial condition, hence
		\begin{equation}
		H(t,x,u) = H(t,x,u_0) = H(t,x,\tilde{u}) .
		\label{memory equal}
		\end{equation}
		Let us set $v:=u-\tilde{u}$.
		
\subsection{Step 1: the linear problem satisfied by $v$} \hfill

Substracting the problem \eqref{complete eq} satisfied by $u$ with the one satisfied by $\tilde{u}$, we obtain that the function $v$ verifies
		\begin{equation}
v_t - (\rho (x) v_x)_x = r(t)q(x)\beta(u)-r(t)\tilde{q}(x)\beta(\tilde{u})-(\varepsilon(u)|u|^3u-\varepsilon(\tilde{u})|\tilde{u}|^3\tilde{u})
		\label{eq v}
		\end{equation}
		for all $t \in (0,T)$ and $x \in (-1,1)$.
We linearize \eqref{eq v} thanks to the regularity of $\beta$ (in Sellers type models) and of $\varepsilon$, defining
		\begin{align*}
		\mu_1(u,\tilde u)&:=
		\begin{cases}
		\dfrac{\varepsilon(u)|u|^3u-\varepsilon(\tilde{u})|\tilde{u}|^3\tilde{u}}{u-\tilde{u}} & u \neq \tilde{u}\\
		\frac{\partial}{\partial u}(\varepsilon(u)|u|^3u) & u = \tilde{u}
		\end{cases}\\
		\intertext{and}
		\mu_2(u,\tilde u)&:=
		\begin{cases}
		\dfrac{\beta (u) - \beta(\tilde{u})}{u-\tilde{u}} & u \neq \tilde{u}\\
		\frac{\partial}{\partial u}(\beta(u)) & u = \tilde{u}
		\end{cases} .
		\end{align*}
		Let us add and substract $r(t)\tilde{q}(x)\beta(u) $ and then replace $\mu_1$ and $\mu_2$ in \eqref{eq v}, so we obtain the following linear equation with respect to $v$:
		\begin{equation}
v_t - (\rho (x) v_x)_x = r(t)\tilde{q}(x)\mu_2(u,\tilde u) v - \mu_1(u,\tilde u) v + r(t)\beta(u)(q(x)-\tilde{q}(x)) .
		\label{linear}
		\end{equation}


\subsection{Step 2: $q=\tilde q$ on $(-1,x_0)$} \hfill

We define the largest interval $[y_1,x_0]$ where $q \equiv \tilde{q} $ and we want to prove that $y_1 = -1$.
		
		Let us set
		\begin{equation*}
		\mathcal{A}^-:=\{x \leq x_0 : (q - \tilde{q})(y) \equiv 0 \quad \forall y \in [x,x_0] \}
		\end{equation*}
If $\mathcal{A}^- \neq \emptyset$, we consider 
		$$y_1:= \inf \mathcal{A}^- ,$$
and if $\mathcal{A}^- = \emptyset$, we consider 
		$$y_1:=  x_0,$$
so that in any case, we know that if $y_1 > -1$, and if $\eta >0$ is such that $y_1 - \eta >-1$, then there exists $y_2 \in (y_1-\eta,y_1)$ such that $q(y_2) \neq \tilde q (y_2)$.

		To show that $y_1 = -1$, we argue by contradiction, so let us assume that $y_1>-1$.
		
	 	\textit{STEP 2.1} First of all we want to prove that there exists $ t_1 \in [0,T) $ and $y_2 \in (-1,y_1) $ such that
		$v(t,y_2)$ never vanishes on $(0,t_1) $.
Since $q, \tilde{q} \in \mathcal{M}$, we have that $q-\tilde{q} \in \mathcal{M}$. It follows that there exists $y'_1<y_1$ such that $q-\tilde{q}$ is analytic on $[y'_1,y_1]$, hence is constantly equal to zero or has only a finite number of zeros. Since we already noted that the definition of $y_1$ implies that $q-\tilde q$ cannot be constantly equal to zero on some interval $(y_1 - \eta, y_1]$ (with $\eta >0$), then $q-\tilde{q}$ has only a finite number of zeros in $[y'_1,y_1]$, and this implies that 
		\begin{equation}
		\exists \mbox{ } y_2 \in (y'_1,y_1) \mbox{ s.t } (q-\tilde{q})(x) \neq 0 \text{ for all } x \in [y_2,y_1) .
		\label{rho primo}
		\end{equation}
Without loss of generality, we can assume that
		\begin{equation}
		(q-\tilde{q})(x)>0 \quad \forall x \in [y_2,y_1) .
		\label{bigger than 0}
		\end{equation}
		Let us notice that since $u$ and $\tilde{u}$ have the same initial condition, we have that $v(0,x)=0$ for all $x \in [-1,1]$ and that $v_x(0,x)=0$. Using this remark and computing \eqref{linear} at $t = 0$ and $x = y_2$,
we obtain
		\begin{equation*}
		v_t (0,y_2) = r(0)\beta(u_0)(q-\tilde{q})(y_2) ,
		\end{equation*}
and then \eqref{bigger than 0} implies that $v_t (0,y_2)>0$. Therefore, since $v(0,y_2) = 0$, we have that there exists some time $t_1 \in (0,T)$ such that
		\begin{equation*}
		v(t,y_2)>0 \quad \forall t \in (0,t_1) .
		\end{equation*}
		Moreover, the assumption \eqref{assumption u} implies that $v(t,x_0) =0$ for all $t \in (0,T)$.
		
		\textit{STEP 2.2} Using the strong maximum principle and the Hopf's Lemma, we are going to prove that the assumption $y_1>-1$ leads to a contradiction.
		Consider
		\begin{align*}
		 K:= \max_{t \in [0,t_1], x \in [y_2,x_0]} -\mu_1 (u(t,x), \tilde u (t,x)) + r(t) \tilde{q}(x)  \mu_2 (u(t,x), \tilde u (t,x)) :
		\end{align*}
$K$ is chosen so that
$$ R(t,x):= -\mu_1 (u(t,x), \tilde u (t,x)) + r(t) \tilde{q} (x) \mu_2 (u(t,x), \tilde u (t,x)) - K \leq 0 .$$
Let also define
$$ w(t,x):=v(t,x)e^{-Kt} .$$
		Using \eqref{linear}, we observe that
$$		w_t - (\rho(x) w_x)_x - R(t,x) w 
		= r(t)\beta(u)(q(x)-\tilde{q}(x)) \, e^{-Kt} .
$$
Since 
$$ \forall x\in [y_2,x_0], \quad q(x)-\tilde{q}(x) \geq 0,$$
we obtain that $w$ satisfies
		\begin{equation*}
		\begin{cases}
		w_t - (\rho(x) w_x)_x - R(t,x)w  \geq 0 &(t,x) \in (0,t_1) \times (y_2,x_0) , \\
		w(0,x)=0 & x \in [y_2,x_0] , \\
		w(t,x_0) = 0 & t \in (0,t_1) , \\
		w(t,y_2) > 0 & t \in (0,t_1)
		\end{cases}
		\label{eq w}
		\end{equation*}
		where the second condition follows from the initial conditions of \eqref{complete eq}, the third from the assumption \eqref{assumption u}, and the last from Step 2.1.
		
		Let us notice that, since $ [y_2,x_0] \subset (-1,1) $, we can apply the strong maximum principle (Chapter 3 of \cite{protter2012maximum}).
		It implies that
		\begin{equation}
		w(t,x)>0 \qquad \forall (t,x) \in (0,t_1) \times (y_2,x_0) .
		\label{property w}
		\end{equation}
Moreover, since $w(t,x_0)=0$ and $x_0 \neq 1$, we can apply the Hopf's Lemma which implies that
		\begin{equation*}
		w_x (t,x_0) < 0 \qquad \forall t \in (0,t_1)
		\end{equation*}
		It follows that $u_x (t,x_0) < \tilde{u} _x (t,x_0)$ for all $t \in (0,t_1)$ which contradicts the second assumption in \eqref{assumption u}.
		
		As a consequence, the assumption $y_1>-1$ is false and this implies that $y_1 = -1$, therefore $q \equiv \tilde{q}$ on $(-1,x_0]$.

\subsection{Conclusion} \hfill
		
		The proof is equivalent for $[x_0,1)$, and one can show that $q \equiv \tilde{q} $ on $(-1,1)$.
		
		The uniqueness result of Theorem \ref{thm-Sellers global existence} implies that $u \equiv \tilde{u}$ on $[0,+\infty) \times (-1,1)$.\qed
		

\section{Proof of Theorem \ref{thm-inv-stab}}
\label{inverse problem2}

We follow the strategy used in \cite{TORT2012683, Sellers-PM-JT-JV} (that was adapted from the method introduced in Imanuvilov-Yamamoto, to study the Sellers case), focusing on the changes brought by the memory term. The proof of the stability result is decomposed in several steps, we give the main intermediate results and we will refer to \cite{TORT2012683, Sellers-PM-JT-JV} for some details.

Remember that $u$ and $\tilde u$ are the solutions of \eqref{stab-eq-u} and \eqref{stab-eq-utilde}. 
Thanks to the assumptions, we can assume that $T<\delta$ without loss of generality.

\subsection{Step 1: the problem solved by the difference $w:=u-\tilde u$} \hfill

Clearly the difference 
\begin{equation}
\label{stab-def-w}
w := u - \tilde u 
\end{equation}
satisfies the problem

\begin{equation}
\label{stab-eq-w}
\begin{cases}
w_t - (\rho w_x)_x = K^* + K + \tilde K + K^h , \quad t>0, x\in (-1,1), \\
\rho w_x = 0 , \quad x= \pm 1 , \\
w(s,x)= u_0 (s,x) - \tilde u_0 (s,x), \quad s\in [-\tau,0], x\in (-1,1) ,
\end{cases}
\end{equation}
where the source terms $K^*$, $K$, $\tilde K$ and $K^h$ are defined by
\begin{equation}
\label{stab-def-K*}
K^* := r(t) (q(x)-\tilde q(x)) \beta (u),
\end{equation}

\begin{equation}
\label{stab-def-K}
K := r(t) \tilde q (x) (\beta (u) - \beta (\tilde u)),
\end{equation}

\begin{equation}
\label{stab-def-Kt}
\tilde K := - \varepsilon (u) \vert u \vert ^3 u + \varepsilon (\tilde u) \vert\tilde  u \vert ^3 \tilde u,
\end{equation}

\begin{equation}
\label{stab-def-Kh}
K^h := f(H)- f(\tilde  H) .
\end{equation}
Note that compared to \cite{TORT2012683, Sellers-PM-JT-JV}, our goal is similarly to estimate from above $K^*$ and the only difference lies in the presence of $K^h$. However, since we assumed that the memory kernel $k$ is supported in $[-\tau, -\delta]$, and that $T<\delta$, it is clear (as in the previous section) that the memory terms $H$ and $\tilde H$ are directly determined from the initial conditions $u_0$ and $\tilde u_0$.


\subsection{Step 2: a useful property of the source term $K^*$} \hfill

We claim that $K^*$ satisfies the following property (classical in that question of determining a source term):

\begin{equation}
\label{stab-estim-K*}
\exists C_0 >0 \quad \text{ s.t.} \quad \forall t \in (0,T), \forall x \in (-1,1), \quad \Bigl \vert \frac{\partial K^*}{\partial t} (t,x) \Bigr \vert \leq C_0 K^* (T',x)  .
\end{equation}
Since
$$ K^* _t := r'(t) (q(x)-\tilde q(x)) \beta (u) + r(t) (q(x)-\tilde q(x)) \beta ' (u) u_t,
$$
\eqref{stab-estim-K*} is an easy consequence of the following regularity result:

\begin{Lemma}
\label{lem-reg-u_t}
Under the regularity assumptions of Theorem \ref{thm-inv-stab}, the solution $u$ of 
\eqref{stab-eq-u} satisfies: $u_t \in L^\infty ((0,T)\times I)$, and more precisely, there exists $C(T,M,M')>0$ such that, for all $u_0 \in \mathcal U^{(loc)} _M$, for all $q\in \mathcal Q^{(loc)} _{M'}$, we have 
$$ \Vert u_t \Vert _{L^\infty ((0,T) \times I)} \leq C(T,M,M') .$$
\end{Lemma}
Lemma \ref{lem-reg-u_t} can be proved as Theorem 3.4 and Corollary 3.1 of \cite{TORT2012683}, noting that the additive memory term satisfies
$$
\forall t\in (0,T), \quad  H(t,x)
=  \int _{t-\tau} ^{t-\delta} k(s-t,x) u(s,x) \, ds
= \int _{t-\tau} ^{t-\delta} k(s-t,x) u_0(s,x) \, ds ,
$$
hence 
$$ H_t = k(-\tau,x) u_0(t-\tau,x) - k(-\delta,x) u_0(t-\delta,x) - \int _{t-\tau} ^{t-\delta} k_t(s-t,x) u_0(s,x) \, ds ,$$
which is bounded from the regularity assumptions of Theorem \ref{thm-inv-stab}.


\subsection{Step 3: a Carleman estimate on the problem solved by $z:=w_t$} \hfill

Consider $z:=w_t$. It is solution of the following problem
\begin{equation}
\label{stab-eq-z}
\begin{cases}
z_t - (\rho z_x)_x = K^* _t + K_t + \tilde K_t + K^h _t , \quad t>0, x\in (-1,1), \\
\rho z_x = 0 , \quad x= \pm 1 ,
\end{cases}
\end{equation}
and we can apply standard Carleman estimates for such degenerate operator (see \cite{sicon2008, memoire}): choosing 
\begin{itemize}
\item $\theta: (t_0,T) \to \Bbb R_+ ^*$ smooth, strictly convex, such that
$$ \theta (t) \to + \infty \quad \text{ as } t\to t_0 ^+ \text{ and as } t \to T^- ,$$
and $\theta ' (T')=0$ such that $T'$ is the point of global minimum,
\item and $p: (-1,1) \to [1,+\infty)$ well designed with the respect to the degeneracy (see \cite{sicon2008} in the typical degenerate case, and \cite{TORT2012683} for an explicit construction in the case of the Sellers model),
\end{itemize}
and considering
$$ \sigma (t,x) = \theta (t) p(x) ,$$
and $R>0$ large enough,
the following Carleman estimate holds true, see Theorem 4.2 in \cite{TORT2012683}:
\begin{multline}
\label{Carleman-z}
\int _{t_0} ^T \int _{-1} ^1 \Bigl( R^3 \theta ^3 (1-x^2) z^2 + R\theta (1-x^2) z_x ^2 + \frac{1}{R\theta} z_t ^2 \Bigr)  e^{-2R\sigma}
\\
\leq C_1 \Bigl( \int _{t_0} ^T \int _{-1} ^1 (K^* _t + K_t + \tilde K_t + K^h _t) ^2 e^{-2R\sigma} + \int _{t_0} ^T \int _{a} ^b R^3 \theta ^3 z^2 e^{-2R\sigma} \Bigr) .
\end{multline}
Then,
\begin{itemize}
\item  since $u_t$ is bounded, we immediately obtain that
$$ \vert K_t \vert + \vert \tilde K _t \vert \leq c (\vert w\vert + \vert z \vert) ,$$
and this allows to show that
\begin{equation}
\label{estim-KKt}
\int _{t_0} ^T \int _{-1} ^1 (K_t ^2 + \tilde K_t ^2 ) e^{-2R\sigma} 
\leq c\Bigl(  \Vert w(T') \Vert ^2 _{L^2 (I)}  + \int _{t_0} ^T \int _{-1} ^1  z^2 e^{-2R\sigma} \Bigr);
\end{equation}
\item for the memory term: using the explicit form given by the initial conditions, we immediately have 
$$ K^h _t  = f'(H)H_t - f'(\tilde H) \tilde H _t 
= f'(H) (H_t - \tilde H_t) + (f'(H) - f'(\tilde H)) \tilde H_t ,
$$
hence
$$ \vert K^h _t \vert 
\leq c \Bigl( \vert H_t - \tilde H_t \vert + \vert H - \tilde H \vert  \Bigr) $$
and then
\begin{equation}
\label{estim-Kh}
\int _{t_0} ^T \int _{-1} ^1 K_h ^2 e^{-2R\sigma} 
\leq c \Vert u_0 - \tilde u_0 \Vert _{C([-\tau,0],V)} ^2 .
\end{equation}

\end{itemize}
Using \eqref{estim-KKt} and \eqref{estim-Kh} in the Carleman estimate \eqref{Carleman-z},
we obtain
\begin{multline*}
\int _{t_0} ^T \int _{-1} ^1 \Bigl( R^3 \theta ^3 (1-x^2) z^2 + R\theta (1-x^2) z_x ^2 + \frac{1}{R\theta} z_t ^2 \Bigr)  e^{-2R\sigma}
\\
\leq C \Bigl( \int _{t_0} ^T \int _{-1} ^1 (K^* _t)^2 e^{-2R\sigma} + \int _{t_0} ^T \int _{a} ^b R^3 \theta ^3 z^2 e^{-2R\sigma} \Bigr) 
\\
+ C \Bigl(  \Vert w(T') \Vert ^2 _{L^2 (I)}  + \int _{t_0} ^T \int _{-1} ^1  z^2 e^{-2R\sigma} \Bigr) + c \Vert u_0 - \tilde u_0 \Vert _{C([-\tau,0],V)} ^2 .
\end{multline*}
Absorbing the term $  \int _{t_0} ^T \int _{-1} ^1  z^2 e^{-2R\sigma}$ in the left-hand side (this is classical, using Hardy type inequalities, see \cite{sicon2008, TORT2012683}), we finally obtain

\begin{multline}
\label{step3}
I_0 := \int _{t_0} ^T \int _{-1} ^1 \Bigl( R^3 \theta ^3 (1-x^2) z^2 + R\theta (1-x^2) z_x ^2 + \frac{1}{R\theta} z_t ^2 \Bigr)  e^{-2R\sigma}
\\
\leq C \int _{t_0} ^T \int _{-1} ^1 (K^* _t)^2 e^{-2R\sigma} + C \int _{t_0} ^T \int _{a} ^b R^3 \theta ^3 z^2 e^{-2R\sigma} \Bigr) 
\\
+ C   \Vert w(T') \Vert ^2 _{L^2 (I)}   + C \Vert u_0 - \tilde u_0 \Vert _{C([-\tau,0],V)} ^2 .
\end{multline}


\subsection{Step 4: an estimate from above} \hfill

Using Step 2, we see that
\begin{multline*}
\int _{t_0} ^T \int _{-1} ^1 (K^* _t)^2 e^{-2R\sigma} \, dx \, dt 
\leq 
C_0 ^2 \int _{t_0} ^T \int _{-1} ^1 K^*(T')^2 e^{-2R\sigma} \, dx \, dt 
\\
C_0 ^2 \int _{-1} ^1 K^*(T')^2 \Bigl( \int _{t_0} ^T e^{-2R\sigma} \, dt \Bigr) \, dx ,
\end{multline*}
and it is classical (see Imanuvilov-Yamamoto \cite{Imanuvilov}, equation (3.17)) that
$$  \int _{t_0} ^T e^{-2R\sigma} \, dt  = o (e^{-2R\sigma (T')})
\quad \text{ as } R\to +\infty .$$
(This is due to the convexity of the function $\theta$, that attains its minimum at $T'$.)
Hence 
\begin{equation}
\label{step4}
\int _{t_0} ^T \int _{-1} ^1 (K^* _t)^2 e^{-2R\sigma}\, dx \, dt  = o\Bigl( \int _{-1} ^1 K^*(T')^2 e^{-2R\sigma (T')} \, dx \Bigr) .
\end{equation}

\subsection{Step 5: an estimate from below} \hfill

As in \cite{TORT2012683, Sellers-PM-JT-JV}, we have
\begin{equation}
\label{step5}
\int _{-1} ^1 z(T')^2 e^{-2R\sigma (T')} 
\leq c I_0 .
\end{equation}


\subsection{Step 6: conclusion} \hfill

Now we are in position to conclude: using the equation in $w$:
$$ K^* (T') = z (T') - (\rho w_x)_x (T') - K(T') - \tilde K (T') - K^h (T') ,$$
hence
\begin{multline*}
\int _{-1} ^1 K^*(T')^2 e^{-2R\sigma (T')}
\\
\leq C \Bigl( \int _{-1} ^1 ( z (T')^2 + (\rho w_x)_x (T')^2 + K(T')^2 + \tilde K (T')^2 + K^h (T') ^2) e^{-2R\sigma (T')} \Bigr) ;
\end{multline*}
now note that
$$ \vert K(T') \vert \leq c \vert w (T') \vert, \quad  \vert \tilde K(T') \vert \leq c \vert w (T') \vert ,$$
and 
$$ \int _{-1} ^1 K^h(T') ^2 e^{-2R\sigma (T')} \leq C \Vert u_0 - \tilde u_0 \Vert _{C([-\tau,0],V)} ^2 ;$$
then using \eqref{step5} , \eqref{step3} and \eqref{step4}, we obtain that
\begin{multline*}
\int _{-1} ^1 K^*(T')^2 e^{-2R\sigma (T')}
\leq C
\Bigl( \Vert w( T') \Vert _{D(A)}^{2}
\\
+ \Vert w_t \Vert _{L^{2}((t_0,T)\times (a,b))}^{2}
+ \Vert u_0 - \tilde u_0 \Vert _ {C([-\tau,0]; V)} ^2 \Bigr) .
\end{multline*}
Looking to the form of $K^*$, we obtain \eqref{PISstab1var} . \qed


\section{Proof of Theorem \ref{thm-Budyko global existence}}
\label{sec-proof-wp-B}


\subsection{The strategy to prove Theorem \ref{thm-Budyko global existence}} \hfill

The method is usual:
\begin{itemize}
\item first we approximate the coalbedo $\beta $ by a sequence of smooth functions $\beta _j$,
\item then we consider the approximate problem associated to $\beta _j$, and we denote $u_j$ its (unique) solution,
\item we obtain suitable assumptions on $u_j$, and we pass to the limit, and we prove that we obtain a solution $u_\infty$ of the original problem.
\end{itemize}

Let us be more precise:
first, given $j\geq 1$, we consider a function $ \beta _j: \Bbb R \to \Bbb R$ which is of class $C^1 $, nondecreasing, and
$$ \begin{cases}
\beta _j (u) = a_i , \quad u \leq \bar u - \frac{1}{j} , \\
\beta _j (u) = a_f , \quad u \geq \bar u + \frac{1}{j} .
\end{cases} $$

Then we can consider the approximate problem
\begin{equation}
\label{stab-eq-uj}
	\begin{cases}
	u_t - (\rho (x) u_x)_x = r(t)q(x)\beta _j (u) - (a+bu) + f(H), \quad t>0,  x \in I,\\
	\rho (x) u_x = 0, \quad x = \pm 1 ,\\
	u(s,x) = u_0(s,x), \quad s \in [-\tau,0], x\in I ,
	\end{cases}
\end{equation}
which is of Sellers type. Hence, since $u_0 \in C([-\tau,0];V)$, $u_0(0) \in D(A) \cap L^\infty (I)$, 
Theorem \ref{thm-Sellers global existence} ensures us that, given $T>0$, the problem \eqref{stab-eq-uj} has a unique solution $u_j$
such that, for all $T>0$, we have
\begin{equation}
\label{reg-uj}
u_j \in H^1(0,T;L^2(I)) \cap L^2(0,T;D(A)) \cap C([-\tau,T];V) .
\end{equation}
We will denote
\begin{equation}
\label{def-gammauj}
\gamma _j (t,x) := r(t)q(x)\beta _j (u_j(t,x)) - (a+b u_j(t,x)) + f(H_j(t,x)),
\end{equation}
where of course
$$ H_j (t,x) = \int _{-\tau} ^{0} k(s,x), u_j (t+s,x) \, ds .$$
In order to pass to the limit $j\to \infty$ in the approximate problem \eqref{stab-eq-uj}, we need some estimates on 
$u_j$ and $\gamma _j$. We provide them in the following lemmas:

\begin{Lemma}
\label{lem-estim-uj}
The family $(u_j)_j$ is relatively compact in $C([0,T];L^2 (I))$.
\end{Lemma}

\begin{Lemma}
\label{lem-estim-gammaj}
The family $(\gamma_j)_j$ is weakly relatively compact in $L^2 (0,T;L^2 (I))$.
\end{Lemma}

Assume that Lemmas \ref{lem-estim-uj} and \ref{lem-estim-gammaj} hold true. Then, we can extract from $(u_j, \gamma_j)_j$
a subsequence $(u_{j'}, \gamma_{j'})_{j'}$ such that
$$ u_{j'} \to u_\infty \text{ in } C([0,T];L^2 (I)) \quad \text{ and } \quad
\gamma _{j'} \rightharpoonup \gamma _\infty \text{ in } L^2 (0,T;L^2 (I)) .$$
 
Then we prove that
\begin{Lemma}
\label{lem-pbm-limite}
The functions $u_\infty$ and $\gamma _\infty$ satisfy the following formula
$$ \forall t \in [0,T], \quad u_\infty (t) = e^{tA} u_0 (0) + \int _0 ^t e^{(t-s)A} \gamma _\infty (s) \, ds .$$
Moreover,
$$ u_\infty \in H^1(0,T;L^2(I)) \cap L^2(0,T;D(A)) \cap C([-\tau,T];V) ,$$
and 
$$ \forall s \in [-\tau,0], \quad u_\infty (s) = u_0 (s) .$$
\end{Lemma}
Finally, it remains to prove that $u_\infty$ is solution of the original Byudyko problem, and so we have to prove that 
$\gamma _\infty$ satisfies the set inclusion. This is the object of the last lemma:

\begin{Lemma}
\label{lem-2nd-membre}
The limit function $\gamma _\infty$ satisfies the set inclusion:
\begin{multline}
\label{lader?}
\gamma _\infty (t) (x) 
+ (a+b u_\infty (t,x)) - f(\int _{-\tau} ^0 k(s,x) u_\infty (t+s,x) \, ds )
\\
\in r(t) q(x) \mathcal \beta (u_\infty (t,x)) 
\text{ a.e.} (t,x) \in (0,T) \times I .
\end{multline}
Therefore the function $u_\infty$ is solution of the original Byudyko problem \eqref{stab-eq-u-B}.
\end{Lemma}

Finally, note that using the Cantor diagonal process, we can extract subsequences that converge in the same way in all compact subsets of $[0,+\infty)$, hence $u_\infty$ is well defined on $[0,+\infty)$.


\subsection{Proof of Lemma \ref{lem-estim-uj}.} \hfill

First note that from Theorem \ref{maximum}, we already know that the family $(u_j)_j$ belongs to $L^\infty ([0,+\infty) \times (-1,1))$ and is uniformly bounded in this space.

To prove that the sequence $(u_j)_j$ is relatively compact in $C([0,T];L^2 (I))$, we are going to apply the Ascoli-Arzela theorem (we refer, e.g., to \cite{vrabie1987compactness} Theorem 1.3.1 p. 10): we have to prove that

\begin{itemize}
\item $(u_j)_j$ is equicontinuous, that is: 
$$ \sup _j \sup _{[0,T]} \Vert u_j (t+h) - u_j(t)  \Vert _{L^2(I)} \to 0 \quad \text{ as } h\to 0 ,$$
\item and that, for all $t\in [0,T]$, the set of traces $\{u_j (t), j\geq 1 \}$ is relatively compact in $L^2 (I)$. 
\end{itemize}

The result on the set of traces $\{u_j (t), j\geq 1 \}$ follows from a regularity result: we know from \eqref{reg-uj} that
$u_j (t) \in V$. Moreover, it follows from the proof of Proposition \ref{prop-local existence} that there is some $M^*$
independent of $j$ such that
$$ \sup _{t\in [0,T]} \Vert u_j (t) \Vert _V  \leq M^* .$$
Since the injection of $V$ into $L^2 (I)$ is compact, we obtain that the set of traces $\{u_j (t), j\geq 1 \}$ is relatively compact in $L^2 (I)$.

Next, from the problem satisfied by $u_j$, we have the following integral representation formula
$$ u_j (t) = e^{tA} u_0 (0) + \int _0 ^t e^{t-s)A} \gamma _j (s) \, ds ,$$
hence (if $h>0$)
\begin{multline*}
u_j (t+h) - u_j(t) 
\\
= \Bigl( e^{(t+h)A} u_0 (0) - e^{tA} u_0 (0) \Bigr) + \Bigl( \int _0 ^{t+h} e^{(t+h-s)A} \gamma _j (s) \, ds
- \int _0 ^t e^{(t-s)A} \gamma _j (s) \, ds \Bigr)
\\
= \Bigl( (e^{(t+h)A} - e^{tA}) u_0 (0) \Bigr) 
+ \Bigl( \int _t ^{t+h} e^{(t+h-s)A} \gamma _j (s) \, ds \Bigr)
\\
+  \Bigl( \int _0 ^{t} ( e^{(t+h-s)A} -e^{(t-s)A}) \gamma _j (s) \, ds \Bigr) .
\end{multline*}
We estimate these three terms:
\begin{itemize}
\item first, by usual estimates
$$ \Bigl \Vert e^{(t+h)A} u_0 (0) - e^{tA} u_0 (0) \Bigr \Vert _{L^2(I)}
\leq c \vert h \vert \Vert u_0 (0) \Vert _{D(A)} ;$$

\item next,
$$ \Bigl \Vert \int _t ^{t+h} e^{(t+h-s)A} \gamma _j (s) \, ds \Bigr \Vert 
\leq \int _t ^{t+h} \Vert \gamma _j (s)\Vert _{L^2(I)}  \, ds \leq c \vert h \vert ;$$

\item finally, as we did previously, 
\begin{multline*}
\Bigl \Vert \int _0 ^{t} ( e^{(t+h-s)A} -e^{(t-s)A}) \gamma _j (s) \, ds \Bigr \Vert _{L^2(I)}
= \Bigl \Vert (e^{hA} - Id) \int _0 ^{t} e^{(t-s)A} \gamma _j (s) \, ds \Bigr \Vert _{L^2(I)}
\\
= \Bigl \Vert \Bigl( \int _0 ^h A e^{\sigma A} \, d\sigma \Bigr) \Bigl( \int _0 ^{t} e^{(t-s)A} \gamma _j (s) \, ds \Bigr) \Bigr \Vert _{L^2(I)}
\\
= \Bigl \Vert \Bigl( \int _0 ^h (-A)^{1/2} e^{\sigma A} \, d\sigma \Bigr) \Bigl( \int _0 ^{t} (-A)^{1/2} e^{(t-s)A} \gamma _j (s) \, ds \Bigr) \Bigr \Vert _{L^2(I)}
\\
\leq \Bigl( \int _0 ^h ||| (-A)^{1/2} e^{\sigma A} ||| _{\mathcal L (L^2(I))} \, d\sigma \Bigr)
\Bigl( \int _0 ^{t} ||| (-A)^{1/2} e^{(t-s)A} ||| _{\mathcal L (L^2(I))} \Vert \gamma _j (s) \Vert _{L^2(I)} \, ds \Bigr)
\\
\leq c \Bigl( \int _0 ^h \frac{1}{\sqrt{\sigma}} \, d\sigma \Bigr)
\Bigl( \int _0 ^{t} \frac{1}{\sqrt{t-s}} \, ds \Bigr)
\leq c' \sqrt{h} .
\end{multline*}
\end{itemize} 
These three estimates prove that
$$ \sup _j \sup _{[0,T]} \Vert u_j (t+h) - u_j(t)  \Vert _{L^2(I)} = O (\sqrt{\vert h \vert })
\quad \text{ as } h \to 0 .$$ 
Hence the family $(u_j)_j$ is equicontinuous in $C([0,T];L^2 (I))$. Therefore the Ascoli-Arzela theorem 
allows us to say that the family $(u_j)_j$ is relatively compact in $C([0,T];L^2 (I))$. \qed


\subsection{Proof of Lemma \ref{lem-estim-gammaj}.} \hfill

Since the family $(u_j)_j$ is uniformly bounded in $L^\infty ((0,T),\times (-1,1))$, we deduce that the same property holds true pour the family $(\gamma _j)_j$, hence also in the space $L^2 (0,T;L^2 (I))$. Hence 
the family $(\gamma_j)_j$ is weakly relatively compact in $L^2 (0,T;L^2 (I))$. \qed

\subsection{Proof of Lemma \ref{lem-pbm-limite}.} \hfill

We start from the integral formula
\begin{equation}
\label{form-uj}
u_j (t) = e^{tA} u_0 (0) + \int _0 ^t e^{(t-s)A} \gamma _j (s) \, ds .
\end{equation}
Choose $\xi \in L^2(I)$, and fix $t \in [0,T]$. Then
\begin{multline*}
\langle \xi, \int _0 ^t e^{(t-s)A} \gamma _j (s) \, ds \rangle _{ L^2(I)}
= \int _0 ^t \langle \xi,  e^{(t-s)A} \gamma _j (s)  \rangle _{ L^2(I)} \, ds
\\
= \int _0 ^t \langle e^{(t-s)A} \xi,  \gamma _j (s)  \rangle _{ L^2(I)} \, ds
= \langle z, \gamma _j \rangle _{L^2 (0,T;L^2(I))} 
\end{multline*}
where $z \in L^2 (0,T;L^2(I))$ is defined by
$$  z(s ) := 
\begin{cases} 
e^{(t-s)A} \xi , \quad 0 \leq s \leq t , \\
0 , \quad t \leq s \leq T .
\end{cases} $$
Since $\gamma _{j'} \rightharpoonup \gamma _\infty$ in $L^2 (0,T;L^2 (I))$, we obtain that
$$ \langle z, \gamma _{j'} \rangle _{L^2 (0,T;L^2(I))} \to \langle z, \gamma _\infty \rangle _{L^2 (0,T;L^2(I))}
\quad \text{ as } j' \to \infty .$$
But
\begin{multline*}
\langle z, \gamma _\infty \rangle _{L^2 (0,T;L^2(I))}
= \int _0 ^t \langle z(s) ,  \gamma _\infty (s)  \rangle _{ L^2(I)} \, ds
= \int _0 ^t \langle e^{(t-s)A} \xi,  \gamma _\infty (s)  \rangle _{ L^2(I)} \, ds
\\
= \int _0 ^t \langle \xi,  e^{(t-s)A} \gamma _\infty (s)  \rangle _{ L^2(I)} \, ds
= \langle \xi, \int _0 ^t e^{(t-s)A} \gamma _\infty (s) \, ds \rangle _{ L^2(I)} .
\end{multline*}
From the fact that $u_{j'} \to u_\infty$ in $C([0,T];L^2 (I))$, we deduce from \eqref{form-uj} that
$$  u_\infty (t) = e^{tA} u_0 (0) + \int _0 ^t e^{(t-s)A} \gamma _\infty (s) \, ds . $$
Finally, since $\gamma _\infty \in L^2(0,T;L^2(I))$, and $u_0 (0) \in D(A)$, $u_\infty$ has the regularity claimed in Lemma 
\ref{lem-pbm-limite}. \qed


\subsection{Proof of Lemma \ref{lem-2nd-membre}.} \hfill

It remains to identify the weak limit $\gamma _\infty$. In order to do this, fix $\kappa >0$ and
let us introduce 
\begin{equation}
\label{subset}
 Q _\kappa := \{(t,x) \in (0,T) \times (-1,1) \, | \ \vert r(t) q(x) \vert \geq \kappa \} ,
\end{equation}
and then we define on $Q _\kappa$ the function
\begin{equation}
\label{coalbed}
\mathcal B (t,x ) := \frac{1}{r(t) q(x)} \Bigl( \gamma _\infty (t,x) + (a+b u_\infty) - f(\int _{-\tau} ^{0} k(s,x) u_\infty (t+s,x) \, ds ) \Bigr) .
\end{equation}
Of course this definition is motivated by the fact that, also on $ Q _\kappa$, we have
$$ \beta _j (u_j (t,x)) =  \frac{1}{r(t) q(x)} \Bigl( \gamma _j (t,x) + (a+b u_j) - f(\int _{-\tau} ^{0} k(s,x) u_j (t+s,x) \, ds ) \Bigr) .$$
Hence we immediately have that
$$ \beta _{j'} (u_{j'}) \rightharpoonup \mathcal B \quad \text{in } L^2 (Q _\kappa) \text{ as } j' \to \infty .$$
This already tells us that
$$ a_i \leq \mathcal B (t,x) \leq a_f \quad a.e. \ (t,x) \in Q _\kappa .$$
Indeed, choose $E$ any measurable part of $Q _\kappa$ and $\chi _E$ its characteristic function: then
$$ \iint _E (\mathcal B - a_i ) = \iint _{ Q _\kappa} (\mathcal B - a_i ) \chi _E 
= \lim _{j'\to \infty} \iint _{ Q _\kappa}  (\beta _{j'} (u_{j'}) -a_i) \chi _E  .$$
Since $ \beta _{j'} (u_{j'}) - a_i \geq 0$, the last quantity is nonegative, and hence
$$ \iint _E (\mathcal B - a_i ) \geq 0 .$$
Since this holds true for all $E$, we obtain that $\mathcal B - a_i \geq 0$ on $Q_\kappa$.
In the same way, $ \mathcal B - a_f \leq 0$ on $Q_\kappa$.

Now we conclude by proving that
$$ \begin{cases} 
(t,x) \in \tilde Q _\kappa, \\
u_\infty (t,x) < \bar u 
\end{cases} \quad \implies \quad \mathcal B (t,x) = a_i .$$
For $\eta >0$, we introduce 
$$ D _{\kappa,\eta} := \{ (t,x) \in Q _\kappa \ | \  u_\infty (t,x) \leq \bar u - \eta \} .$$
Up to some subsequence, we can assume that
$$ u_{j'} \to u_\infty \quad a.e. \ (t,x) \in (0,T) \times (-1,1) .$$
If $(t,x) \in D _{\kappa,\eta}$, there exists $j'(t,x)$ large enough such that
$$ \forall j' \geq j'(t,x), \quad u_{j'} (t,x) \leq \bar u - \frac{\eta}{2}. $$
And then, the construction of $\beta _{j'}$ implies that, if $j'$ is large enough, we have
$$ \beta _{j'} (u_{j'} (t,x)) = a_i . $$
Hence
$$  \beta _{j'} (u_{j'}) \to a_i \quad a.e. \ (t,x) \in D _{\kappa,\eta} .$$
Now, consider $\psi \in L^2 (Q _{\kappa})$. We have
$$\iint _{Q _{\kappa}} ( \beta _{j'} (u_{j'})  - \mathcal B) \psi \to 0 \quad \text{ as } j' \to + \infty .$$
Moreover, if $\psi $ is supported in $D_{\kappa,\eta}$, then, since $\beta _{j'} (u_{j'}) \to a_i$ a.e. $D_{\kappa,\eta}$, we deduce from the using the dominated pointwise convergence theorem that
$$ \iint _{Q _{\kappa}} ( \beta _{j'} (u_{j'})  - \mathcal B) \psi  \to \iint _{Q _{\kappa}}  (a_i- \mathcal B) \psi 
\quad \text{ as } j' \to + \infty .$$
Hence
$$  \iint _{Q _{\kappa}}  (a_i- \mathcal B) \psi = 0$$
and this holds true for all $\psi$ supported in $D_{\kappa,\eta}$, therefore
$$ a_i = \mathcal B \text{ on } D_{\kappa,\eta} .$$
Letting $\eta \to 0$, we obtain that 
$$ \mathcal B = a_i \quad \text{ on } \{ (t,x) \in Q _{\kappa} \ | \  u_\infty (t,x) < \bar u \} .  $$
We can proceed in the same way to prove that
$$ \mathcal B = a_f \quad \text{ on } \{ (t,x) \in Q _{\kappa} \ | \  u_\infty (t,x) > \bar u \} .  $$
And finally, letting $\kappa \to 0$, we obtain that, if $r(t)q(x) \neq 0$, we have
\begin{multline*}
\frac{1}{r(t) q(x)} \Bigl( \gamma _\infty (t,x) + (a+b u_\infty) - f(\int _{-\tau} ^{0} k(s,x) u_\infty (t+s,x) \, ds ) \Bigr)
\\
\begin{cases} 
= a_i \quad \text{ if } u_\infty (t,x) < \bar u ,\\
= a_f \quad \text{ if } u_\infty (t,x) > \bar u, \\
\in [a_i, a_f] \quad \text{ in any case} ,
\end{cases}
\end{multline*}
Of course, if $r(t) q(x)= 0$, then
$$ \gamma _{j'} (t,x) = - (a+b u_{j'}(t,x)) + f(H_{j'}),$$
and we deduce that
$$ \gamma _\infty (t,x) = - (a+b u_\infty (t,x)) + f(\int _{-\tau} ^{0} k(s,x) u_\infty (t+s,x) \, ds ) .$$
Hence
\begin{multline*}
\gamma _\infty (t) (x) = r(t) q(x) \mathcal B (t,x) - (a+b u_\infty (t,x) ) + f(H_\infty (t,x)) 
\\
\in r(t) q(x) \mathcal \beta (u_\infty (t,x)) - (a+b u_\infty (t,x)) + f(H_\infty (t,x))
\text{ a.e.} (t,x) \in (0,T) \times I 
\end{multline*}
hence the set inclusion \eqref{lader?} is satisfied. \qed


\bigskip

\noindent {\bf Acknowledgements.} The authors wish to thank K. Fraedrich for a very interesting discussion on the question of Energy Balance Models with Memory.

\end{document}